\documentclass[oneside,10pt]{amsart}
\usepackage{calc}
\usepackage{graphicx} 
\usepackage{tikz}
\usepackage{amssymb}
\usepackage[colorlinks=true]{hyperref}
\usepackage{multirow}
\usepackage{siunitx}

\newcommand{\kk}[1]{\marginpar{\parbox{3.5cm}{\scriptsize \color{blue} \sf KK:\  #1}}}
\renewcommand{\sp}[1]{\marginpar{\parbox{3.5cm}{\scriptsize \color{blue} \sf SP:\  #1}}}

\newcommand{\arxivonly}[1]{\marginpar{\parbox{3.5cm}{\color{blue} \sc arXiv only}} {\color{blue}#1}} \newcommand{\antsonly}[1]{\marginpar{\parbox{3.5cm}{\color{purple} \sc ants only}} {\color{purple}#1}}
\renewcommand{\arxivonly}[1]{#1}\renewcommand{\antsonly}[1]{}

\theoremstyle{plain}
\numberwithin{equation}{section}
\newtheorem{theorem}[equation]{Theorem}
\newtheorem{lemma}[equation]{Lemma}
\newtheorem{algorithm}[equation]{Algorithm}
\newtheorem{proposition}[equation]{Proposition}
\newtheorem{corollary}[equation]{Corollary}

\theoremstyle{definition}
\newtheorem{definition}[equation]{Definition}

\newtheorem{remark}[equation]{Remark}
\newtheorem{example}[equation]{Example}
\newtheorem{condition}[equation]{Condition}
\counterwithin{figure}{section}
\counterwithin{table}{section}

\newcommand{\fixme}[1]{{\color{red}FIXME: #1}}
\newcommand{\checkme}[1]{{\color{blue}CHECKME: #1}}

\newcommand{\bdiv}{{\;\operatorname{div}\;}}

\newcommand{\disc}{{\operatorname{disc}}}
\newcommand{\rank}{{\operatorname{rank}}}
\newcommand{\Gal}{\operatorname{Gal}}
\newcommand{\Aut}{\operatorname{Aut}}
\newcommand{\Res}{\operatorname{Res}}
\newcommand{\Q}{\mathbb{Q}}
\newcommand{\Qp}{\Q_p}
\newcommand{\N}{\mathbb{N}}
\newcommand{\F}{\mathbb{F}}
\newcommand{\Fq}{\F_q}
\newcommand{\Fp}{\F_p}

\newcommand{\Z}{\mathbb{Z}}

\newcommand{\Rep}{\mathcal{R}}

\newcommand{\KK}{K}

\newcommand{\Ksep}{\KK^{\mathrm{sep}}}
\newcommand{\OK}{\mathcal{O}_\KK}
\newcommand{\RK}{\kappa}
\newcommand{\RepK}{\Rep_\KK}
\newcommand{\piK}{\pi_\KK}
\newcommand{\xiK}{\xi_\KK}
\newcommand{\xiRK}{\overline{\xi}_\KK}

\newcommand{\KL}{L}
\newcommand{\OL}{\mathcal{O}_\KL}
\newcommand{\RL}{\lambda} 
\newcommand{\piL}{\pi_\KL} 
\newcommand{\xiRL}{\overline{\xi}_\KL}
\newcommand{\xiL}{\xi_\KL} 

\newcommand{\Lur}{\KL_{\mathrm{ur}}}
\newcommand{\RLur}{\lambda}

\newcommand{\OLur}{\mathcal{O}_{\Lur}}
\newcommand{\RepLur}{\Rep_{\Lur}}
\newcommand{\RepLurs}{\Rep_{\Lur}^\times}

\newcommand{\rampol}{\mathcal{N}} 
\newcommand{\Apol}{{\overline{A}}} 
\newcommand{\Aadd}{\Apol^{+}} 
\newcommand{\Ainv}{\mathcal{A}} 
\newcommand{\Cinv}{\mathcal{C}} 
\newcommand{\cont}[2]{\operatorname{cont}_{#1}\!\left({#2}\right)}
\newcommand{\card}[1]{\# #1} 
\newcommand{\chr}{\chi} 

\newcommand{\Algo}[5]
                {
                \begin{algorithm}[\texttt{#1}] \label{#2}{$\;$}\rm\\[1ex]                   
\mbox{\enspace}\rlap{\rm Input: }\phantom{\rm Output: }
\parbox[t]{\textwidth-\widthof{{\rm Output: }}-\widthof{\enspace\enspace}}{#3}\\                
\mbox{\enspace}{\rm Output: }
\parbox[t]{\textwidth-\widthof{{\rm Output: }}-\widthof{\enspace\enspace}}{#4}\\[-2ex] \parskip0pt
\begin{list}{}{\setlength{\leftmargin}{0pt}}
\item                    #5
\end{list}
\parskip6pt
                \end{algorithm}
                \goodbreak}

\newcommand{\costlog}{\mathsf{L}}

\newcommand{\costchar}{\mathsf{C}}

\newcommand{\costroot}{\mathsf{Z}}
\newcommand{\costmult}[1]{\mathsf{M}\!\left(#1\right)}
\newcommand{\softoh}[1]{\widetilde{O}\!\left(#1\right)}
\newcommand{\softohalone}{\widetilde{O}}
\newcommand{\oh}[1]{{O}\!\left(#1\right)}
\newcounter{savedcounter}

\author[J. Guardia]{Jordi Guàrdia i Rúbies}
\address{Universitat Politècnica de Catalunya}
\email{jordi.guardia-rubies@upc.edu}

\author[J.W. Jones]{John W. Jones}
\address{Arizona State University}
\email{jj@asu.edu}

\author[K. Keating]{Kevin Keating}
\address{University of Florida}
\email{keating@ufl.edu}

\author[S. Pauli]{Sebastian Pauli}
\address{University of North Carolina Greensboro}
\email{s\_pauli@uncg.edu}

\author[D.P. Roberts]{David P.\ Roberts}
\address{University of Minnesota Morris}
\email{roberts@morris.umn.edu}

\author[D. Roe]{David Roe}
\address{Massachusetts Institute of Technology}
\email{roed@mit.edu}
\date{\today}

\title{Distinguished defining polynomials for extensions of $p$-adic fields}

\begin{document}

\begin{abstract}
    We give an algorithm for choosing a distinguished defining polynomial for a $p$-adic field extension.
    This algorithm formed an important ingredient in the recent expansion of the database of $p$-adic fields within the $L$-functions and modular forms database.
\end{abstract}


\maketitle

\section{Introduction}
\label{sec introduction}

It follows from Krasner's Lemma that a $p$-adic field has only finitely many extensions of a given degree.
It is thus natural to seek a list of polynomials $\varphi(x)$ such that each extension is generated by exactly one $\varphi(x)$.
Once the list has been created, looking up an extension defined by some other polynomial $\psi(x)$ requires
an algorithm which matches $\psi(x)$ to one of the chosen $\varphi(x)$ defining an isomorphic extension.
Such an algorithm, \texttt{polredpadic} (Algorithm \ref{alg polredpadic}), is the central result of this paper.
Note that this algorithm also enables the creation of the list of $\varphi(x)$ in the first place by taking the list with repetition from \cite[\S 2.4, 3.3]{families}, running \texttt{polredpadic} on all of them, and then removing duplicates.

The LMFDB \cite{lmfdb} is a massive online database of number-theoretic objects of many different types.
It contains a database of $p$-adic fields, initially created by Jones and Roberts \cite{jrdatabase} and recently extended by the present authors \cite{families}.
The \texttt{polredpadic} algorithm played the central role in enabling the enumeration of the $890{,}111$ degree $16$ and the $314{,}543$ degree $20$ extensions of $\Q_2$.
While the sheer number of fields prohibits practical enumeration in some degrees (there are over $114$ billion degree $32$ extensions of $\Q_2$), the process is now completely automated and straightforward to carry out for any particular degree.
Moreover, the same approach works to enumerate fields with more restricted invariants, such as a specific discriminant or family \cite[\S 1.2]{families}.

Our basic approach rests on work of Krasner \cite{krasner-cluj} as modernized by Monge \cite{Mo}.
In the case where $\KL/\Qp$ is totally ramified, this theory gives a collection of \emph{near-canonical} Eisenstein polynomials, each of which defines $\KL$.
We generalize this approach to the case that the inertia degree is larger than $1$, replacing Eisenstein polynomials with {\O}ystein polynomials (see Section \ref{sec oystein}), and give a process to pick a \emph{distinguished polynomial} among them (see Condition \ref{con main}).

The name \texttt{polredpadic} is inspired by the PARI-GP \cite{pari} function
{\tt polredabs} which computes a (pseudo) canonical
defining polynomial of a number field and enables searching the LMFDB's number field database given an arbitrary irreducible polynomial over $\Q$.

The paper is structured as follows.
In Section \ref{sec notation} we set up notation and describe how we select polynomials defining unramified extensions.
In Section \ref{sec oystein} we define $\nu$-{\O}ystein polynomials for an unramified polynomial $\nu(x)$,
give Algorithm \ref{alg eisenore} for finding an {\O}ystein polynomial generating an arbitrary extension,
and connect {\O}ystein polynomials to Eisenstein polynomials over the maximal unramified subextension.
In Section \ref{sec polygon} we recall the theory of ramification polygons
and generalize them from their original context of Eisenstein polynomials to {\O}ystein polynomials.
In Section \ref{sec residual polynomials} we recall the theory of residual polynomials,
which are attached to the segments of the ramification polygon
and provide an invariant allowing for the division of families into subfamilies (Definition \ref{def res poly class}).
We also define distinguished residual polynomials,
which provide one part of the recipe for selecting a distinguished polynomial for $\KL/\Qp$.

Section \ref{sec reduction} forms the core of the paper, describing how to reduce an {\O}ystein polynomial to a set of near-canonical options
and then how to pick a distinguished polynomial from among them.
We first give the process for Eisenstein polynomials from \cite{Mo},
with modifications to ensure compatibility with the generic polynomial of the family \cite[\S 3.3]{families}.
As part of this process, we recall additive residual polynomials
and generalize them to {\O}ystein polynomials.
The main theorem underpinning the reduction process is Theorem \ref{theo red},
which describes how the minimal polynomial of a root $\alpha$ of a $\nu$-{\O}ystein polynomial $\psi$ changes when making a substitution of the form $\alpha \mapsto \alpha + \gamma\nu(\alpha)^m$.
This theorem enables setting many $p$-adic digits of the coefficients of $\psi$ to $0$ without changing the extension (see Corollary \ref{coro easy}),
and also provides a reduction algorithm for the remaining coefficients (Algorithm \ref{alg polredslope}).
Section \ref{sec reduction} ends by defining a total order on {\O}ystein polynomials, allowing for the selection of a distinguished defining polynomial.
We close with a complexity analysis of our algorithms in Section \ref{sec complexity}\arxivonly{ and several worked examples in Section \ref{sec examples}}.

Our implementation in Magma \cite{magma} of the algorithms in this paper
is available on GitHub at \url{https://github.com/roed314/padic_db} and data is available on the LMFDB at \url{https://www.lmfdb.org/padicField/}. 

We thank the American Institute of Mathematics for supporting this work through an AIM SQuaRE
and the three anonymous referees whose comments have improved the quality of the paper.

\section{Setup and notation}
\label{sec notation}

Let $\KK$ be a finite extension of $\Qp$ with inertia degree $f_\KK$, ramification index
$e_\KK$, and degree $n_\KK = f_\KK e_\KK$.  Let $\piK$ be a uniformizer for $\KK$ and let
$v_\KK:\KK\rightarrow\Z\cup\{\infty\}$ be the valuation on \(\KK\) 
normalized so that $v_\KK(\piK)=1$.
Let $\OK=\{\alpha\in \KK:v_\KK(\alpha)\ge 0\}$ be the
ring of integers of $\KK$
and let $(\piK)=\{\alpha\in \KK:v_\KK(\alpha)> 0\}$ be the maximal
ideal of $\OK$.
For \(\alpha,\beta\in\KK\), we write \(\alpha\sim\beta\) when \(v_\KK(\alpha-\beta)>v_\KK(\alpha)\).

We denote the residue class field of $\KK$ by $\RK=\OK/(\piK)$ and by \(q=\card{\RK} \) its cardinality.
For a generator
\(\xiRK\) of the cyclic group $\RK^\times$
and $\overline{\gamma}\in\RK^\times$ we
let \(\log_{\xiRK}\overline\gamma\) be the smallest \(n\in\N_0=\N\cup\{0\}\)
with \(\xiRK^n=\overline{\gamma}\).

Let \(\RepK\) be a fixed set of representatives of \(\RK\) in \(\OK\)
with \(\{0,1\}\subseteq \RepK\). 

Let \(\varphi\in\OK[x]\) have degree \(n\).  We write
\(\varphi(x)=\sum_{i=0}^n\varphi_i x^i\), and for each
coefficient we write \(\varphi_i=\sum_{i=0}^\infty \varphi_{i,k}\piK^i\), where \(\varphi_{i,k}\in\RepK\).

Let $\Ksep$ be a separable closure of $\KK$ and let $\KL/\KK$ be a finite
subextension of $\Ksep/\KK$, with inertia degree
$f$, ramification index $e=\epsilon p^w$ with
$p\nmid \epsilon$, and degree $n=ef=\epsilon p^wf$. 
We denote by $\Lur/K$ the maximal unramified subextension
of $\KL/\KK$, by $\KL_0/\KK$ its maximal
tame subextension and by \(\RL\) the residue class field of $\KL$.  
Then \([\Lur:\KK]=[\RL:\RK]=f\), $[\KL_0:\Lur]=\epsilon$, and we get the standard tower of subextensions
\begin{equation}
\label{can}
\KK \stackrel{f}{\subseteq} \Lur  \stackrel{\epsilon}{\subseteq} \KL_0  \stackrel{p^w}{\subseteq} \KL.
\end{equation}
Let $\psi$ be a defining polynomial for $\KL/\KK$ and let $\beta$ be a root of $\psi$ such that
$\KL=\KK(\beta)$.  For \(\gamma\in\KK[x]\) we
denote the characteristic polynomial of $\gamma(\beta)$ by 
\(\chr_{\KL/\KK}(\gamma(\beta))=\chr_\psi(\gamma)=\Res_t\left(\psi(t),x-\gamma(t)\right)\).

For \(\rho(x)=\sum_{i=0}^n\rho_i x^i\in\OL[x]\) we call the lower convex hull \(\rampol\) 
of  \[\left\{\left(-i,v_\KL(\rho_i)\right):0\le i\le n\right\}\] 
the \emph{Newton polygon} of \(\rho\).  Its slopes yield the multiset of
valuations of the roots of \(\rho\).  We denote the \(\piL\)-content of \(\rho\) by
\[\cont{\piL}{\rho}=\piL^{\min_{0\le i\le n}v_\KL(\rho_i)}.\]

\subsection*{Unramified extensions}

The standard choice for defining these extensions are Conway polynomials, which were introduced by Parker
\cite[Section 2.8.3]{holt-eick-obrien} and are the default in Magma \cite{magma}, SageMath \cite{sage} and other systems.
Unfortunately Conway polynomials have not been computed for all
primes and extension degrees considered in the \(p\)-adic field database of the LMFDB.
See \cite[Section 7]{luebeck} for the Conway polynomials known in 2023.

For the database, when no Conway polynomial is known we follow the 
procedure from \cite{jrdatabase} which is in the same spirit, but with fewer
restrictions.  
We pick the polynomial whose roots over \(\Fp\) are
primitive of multiplicative order \(p^{f}-1\)
which is smallest according to the lexicographic ordering used in the definition of Conway polynomials:
for \(\nu(x) = x^n + \sum_{i=0}^{n-1} (-1)^{n-i} \nu_i x^i\in\Z[x]\) and
\(\mu(x) = x^n + \sum_{i=0}^{n-1} (-1)^{n-i} \mu_i x^i\in\Z[x]\) with \(0\le \nu_i,\mu_i<p\)
we define \(\overline\nu < \overline\mu\) when there exists $k$ with \(\nu_i = \mu_i\) for all \(i > k\)  
and \(\nu_k < \mu_k\).  The code allows users to make a different choice if desired.

\subsection*{LMFDB}
Following the choice made in Magma and Sage,
in the LMFDB we define 
unramified extensions using polynomials with coefficients in \(\{0,1,\dots,p-1\}\) that
are congruent to the coefficients of the polynomials described above.
By construction a root \(\overline{\xi}\) over \(\F_p\) of such a polynomial \(\overline\nu\)
is a primitive element and is used as the base of discrete logarithms.
For a \(p\)-adic field \(\KK\) with residue field \(\F_{p^{f_\KK}}\) we choose the set of
representatives 
\[
\RepK=\{c_{f_\KK-1}\xi^{f_\KK-1}+\dots+c_1\xi+{c_0} :c_i\in\{0,1,\dots,p-1\}\}.
\]
In particular, if \(f_\KK=1\) we get \(\RepK=\{0,1,\dots,p-1\}\).

\section{{\O}ystein polynomials}\label{sec oystein}
Every totally ramified extension \(\KL/\KK\) of degree \(e\) 
is generated by a root of an Eisenstein polynomial \(\varphi(x)=x^{e}+\sum_{i=0}^{e-1}\varphi_i x^i \in\OK[x]\), where \(v_\KK(\varphi_i)\ge 1\) for \(1\le i \le e-1\) and \(v_\KK(\varphi_0)=1\).   
In the following we consider a generalization of Eisenstein polynomials that
can be used to define any finite extension of a local field.

\begin{definition}
Let $\nu(x)\in\OK[x]$ be monic and irreducible modulo $\piK$.
A monic polynomial $\varphi\in\OK[x]$
is \(\nu\)-{\O}ystein when
\begin{equation}\label{eq phi exp}
\varphi(x)=\nu(x)^e+\sum_{i=1}^{e-1}\varphi^*_i(x)\nu(x)^i+\varphi^*_0(x)
\end{equation}
for some \(e\in\N\), where 
\(\piK\mid \varphi^*_i(x)\) and \(\deg\varphi^*_i<\deg\nu\)
for \(0\le i\le e-1\),
and $\cont{\piK}{\varphi_0^*}=\piK$. 
We call (\ref{eq phi exp}) the \(\nu\)-expansion of \(\varphi\).  
\end{definition} %

We chose the name {\O}ystein to honor {\O}ystein Ore and to avoid confusion 
with Ore polynomials which are elements of Ore extensions in ring theory.
{\O}ystein polynomials are called \emph{regular} polynomials in \cite{ore-newton}
and polynomials in \emph{Eisenstein form} in \cite{ford-pauli-roblot}.

Let \(\varphi(x)=\nu(x)^e+\sum_{i=0}^{e-1}\varphi^*_i(x)\nu(x)^i\in\OK[x]\) be
a $\nu$-{\O}ystein polynomial with \(\overline\nu(x)\ne x\) and set \(f=\deg\nu\).
Let \(\alpha\in\Ksep\) be a root of \(\varphi(x)\) and set \(\KL=\KK(\alpha)\).  Because \(\overline\nu\) is irreducible we have $f_{\KL/\KK}\ge f$
and because \(\overline\nu(\overline\alpha)=0\) we have \(v_\KL(\alpha)=0\). 
The polynomial $\varphi^*(y)=y^e+\sum_{i=0}^{e-1}\varphi^*_i(\alpha)y^i\in\OL[y]$
has the root \(\nu(\alpha)\).  Since \(v_\KK(\varphi_0^*(\alpha))=1\)
and \(v_\KK(\varphi_i^*(\alpha))\ge 1\) the lower convex hull of
\[
\left\{(-e,0),(-e+1,v_\KK(\varphi_{e-1}^*(\alpha))),\dots, (0,v_\KK(\varphi^*_0(\alpha)))\right\}
\]
consists of a single segment of slope \(\frac{1}{e}\).
Thus, for the root \(\nu(\alpha)\) of \(\varphi^*\) we have \(v_\KK(\nu(\alpha))=\frac{1}{e}\) which yields a lower bound for the ramification index of \(\KL/\KK\), namely $e_{\KL/\KK}\ge e$.
Now, since \(e_{\KL/\KK}f_{\KL/\KK}\le ef =\deg\varphi\) we have 
\(e_{\KL/\KK}=e\) and $f_{\KL/\KK}=f$. 
Furthermore, \(\nu(\alpha)\) is a uniformizer of \(\KL\), $\OL=\OK[\alpha]$,
and \(v_\KK(\disc(\KL/\KK))=v_\KK(\disc(\varphi))\).
Over the unramified subextension \(\Lur=\KK[x]/(\nu)\) of \(\KL/\KK\) the polynomial \(\varphi\) factors into \(f=\deg\nu\) 
polynomials of degree \(e\).

Every finite extension of a local field can be generated by an {\O}ystein polynomial:

\begin{lemma}\label{lem make eisenore}
Let \(\KL/\KK\) be finite with inertia degree $f$ and ramification index $e$ and let \(\piL\) a uniformizer of \(\KL\).
Let $\nu\in\OK[x]$ be monic of degree $f$ such that $\overline\nu\in\RK[x]$ is irreducible and let $\alpha\in\OL$ be a root of $\nu(x)-\piL$. 
Then the minimal polynomial 
of $\alpha$ over $K$ is a \(\nu\)-{\O}ystein polynomial.
\begin{proof}
We have $\nu(\alpha)=\piL$, and
$\RK(\overline\alpha)$ is the residue class field of \(\KL/\KK\)
by construction.
Hence the result follows from \cite[Proposition 4.5]{ford-pauli-roblot}.
\end{proof}
\end{lemma}

%
%

\subsection*{The constant coefficient}


\begin{lemma}\label{lem const coeff}
Let \(\KL/\Lur/\KK\) as in \((\ref{can})\) and let \(\alpha\in\KL\) be a root of \(\nu(x)-\piL\).
Write \(\psi(x)=x^e+\sum_{i=0}^{e-1}\psi_i x^i\in\Lur[x]\) for the minimal polynomial of \(\piL\) over $\Lur$ and
\(\varphi(x)=\nu(x)^e+\sum_{i=0}^{e-1}\varphi_i^*(x) \nu(x)^i\) for  the minimal polynomial of \(\alpha\) over $\KK$.
Then \(\varphi_0^*(\alpha)\sim \psi_0\).
\end{lemma}

\begin{proof}
By Lemma \ref{lem make eisenore} \(\varphi\) is a \(\nu\)-{\O}ystein polynomial.
Since \(\nu(\alpha)=\piL\) and \(\psi(\piL)=0\) we have
\(-\psi_0=\nu(\alpha)^e+\sum_{i=1}^{e-1}\psi_i \nu(\alpha)^i\sim \nu(\alpha)^e\).
Also
\(-\varphi_0^*(\alpha)=\nu(\alpha)^e+\sum_{i=1}^{e-1}\varphi_i^*(\alpha) \nu(\alpha)^i\sim  \nu(\alpha)^e\).
\end{proof}
\arxivonly{
Example \ref{ex false friends} illustrates that this match only works for the first digit of the constant coefficient, but not for the other coefficients.}


\subsection*{The OM-algorithm and {\O}ystein polynomials}

The OM-algorithm\footnote{Other similar algorithms, such as Round 4 \cite{ford-pauli-roblot}, can also be used.} \cite{mng-newton} determines the essential $p$-local invariants of a square-free integral polynomial, facilitating both a number of important computations for global fields \cite{guardia-montes-nart} and the $p$-adic factorization of the polynomial \cite{guardia-nart-pauli}. The `O' in OM-algorithm stands for Okutsu and Ore, and the `M' for Mac Lane and Montes.

It is well known that {\O}ystein polynomials are irreducible. In addition, they have the property that their Okutsu invariants are found after only one step of the OM-algorithm. 
\arxivonly{The other polynomials for which this is the case are those of the form
\(\psi(x)=\nu(x)^e+\sum_{i=0}^{e-1}\psi^*_i(x)\nu(x)^i\) which have a root \(\alpha\) such that
\(v_\KK(\nu(\alpha))=\frac{t}{e}\), \(\gcd(t,e)=1\), and \(t\ne1\).}
Ore \cite{ore-newton} asserts that every finite extension of $\Qp$ admits a \emph{regular} 
defining polynomial, which we call an {\O}ystein polynomial. 
Unfortunately, Ore's result is only theoretical, and does not provide a practical method to find a regular defining polynomial for a given extension.  Lemma \ref{lem make eisenore} gives this construction.  

With the OM-Algorithm and Lemma \ref{lem make eisenore} we find a defining {\O}ystein polynomial \(\varphi\in\OK[x]\) of a finite
extension \(\KL/\KK\).
We assume that for each \(f\in\N\) we can look up a monic polynomial \(\nu(x)\in\OK[x]\) of degree \(f\) 
such that \(\overline\nu\) is irreducible over \(\RK\).


\Algo{{\O}ystein}{alg eisenore}
{\(\psi\in\OK[x]\) irreducible}
{\(\nu\in\OK[x]\) monic irreducible over \(\RK\) and\\ 
\(\varphi\in\OK[x]\) \(\nu\)-{\O}ystein such that \(\KK[x]/(\varphi)\cong\KK[x]/(\psi)\)}
{
Denote by \(\beta\) a root of \(\psi\) and write \(\KL=\KK(\beta)\).\\[-1ex]
\begin{enumerate}
\item with the OM-algorithm find:
\begin{itemize}
\item the ramification index \(e\) of \(\KL/\KK\) and \(\Pi\in\KK[x]\) with \(v_\KK(\Pi(\beta))=1/e\)
\item the inertia degree \(f\) of \(\KL/\KK\) 
\end{itemize}
\item\label{alg eisenore conway} let \(\nu\in\OK[x]\) be monic of degree \(f\) with 
\(\overline{\nu}\in\RK[z]\) irreducible and let \(\overline\gamma\in\RL\) 
be a root of \(\nu\)
\item\label{alg eisenore gamma} use the output of the OM-algorithm to find \(\Gamma\in\KK[x]\) with \(\overline{\Gamma(\beta)}=\overline\gamma\) 
\item\label{alg eisenore newton} 
\(\alpha\gets\Gamma(\beta)-\frac{\nu(\Gamma(\beta))-\Pi(\beta)}{\nu'(\Gamma(\beta))}\)
\item return  \(\nu\) and the characteristic polynomial 
\(\varphi=\chr_{\KK(\beta)/\KK}(\alpha)\) of \(\alpha\) 
\end{enumerate}
}

After step (\ref{alg eisenore gamma}) we have \(v_\KK\left(\nu(\Gamma(\beta))-\Pi(\beta)\right) \ge \frac{1}{e}\).
One iteration of Newton lifting in step (\ref{alg eisenore newton}) is sufficient to obtain \(\alpha\in\KK(\beta)\) with 
\(v_\KK\left(\nu(\alpha)-\Pi(\beta)\right) > \frac{1}{e}\).  Hence \(v_{\KL}(\nu(\alpha))=1\), so \(\chr_\psi(\alpha)\) is 
{\O}ystein.

In the LMFDB we choose the polynomials in step (\ref{alg eisenore conway}) to be Conway polynomials or the polynomials from
\cite{jrdatabase} described in Section \ref{sec notation}.

Given a $\nu$-{\O}ystein polynomial \(\varphi(x)=\nu(x)^e + \sum_{i=0}^{e-1}\varphi^*_i(x)\nu(x)^i\in\OK[x]\) 
we can find an Eisenstein polynomial \(\psi\) over \(\Lur=K[x]/(\nu)\) such that \(\Lur[x]/(\psi)= \KK[x]/(\varphi)\) as follows.
Because the polynomial \(\nu(x)\) splits into linear factors over \(\RLur\), 
say \(\overline\nu(x)=(z-\overline\xi_1)\cdots(z-\overline\xi_f)\), 
we have 
\[\varphi(x)\equiv \nu(x)^e  \equiv (x-\xi_1)^e\cdots(x-\xi_f)^e \pmod\piK\] 
where $\xi_i\in\Lur$ is a lifting of $\overline\xi_i$.
So Hensel lifting yields \(f=\deg\nu\) coprime factors of \(\varphi\) of degree \(e\) over \(\Lur\).
Let \(\psi^*\in\OLur[x]\) be the factor of \(\varphi\) that corresponds to the factor \((x-\xi_1)^e\) of \(\nu^e\). 
Then \(\psi(x)=\psi^*(x+\xi_1)\) is Eisenstein and \(\Lur[x]/(\psi) = \KK[x]/(\varphi)\).

\section{Ramification polygons}\label{sec polygon}
Ramification polygons of Eisenstein polynomials were introduced by Krasner \cite{krasner-cluj}, and can be found with slightly varying definitions in later work \cite{scherk,greve-pauli,Mo,pauli-sinclair}. 
They are a formulation of the Herbrand invariant of an extension of
local fields; see \cite[Section 2]{families} for various perspectives on these.  
Here we generalize the notion of ramification polygons to {\O}ystein polynomials.

\begin{definition}
Let \(\varphi(x)=\nu(x)^e+\sum_{i=0}^{e-1}\varphi^*_i(x) \nu(x)^i \in\OK[x]\) be \(\nu\)-{\O}ystein and let \(\alpha\in\Ksep\) be one of its roots.  We say that 
\[
\rho(x)=\frac{\varphi\left(\nu(\alpha)x +\alpha\right)}{\nu(\alpha)^e}
\] 
is the \emph{$\nu$-ramification polynomial} of \(\varphi\).  We say that
the Newton polygon \(\rampol(\varphi)\) of \(\rho\) is the 
\emph{ramification polygon} of \(\varphi\); note that \(\rampol(\varphi)\) does not depend on $\nu$.
\end{definition}
If we write \(
\rho(x)=\sum_{i=0}^{n}\rho_i x^i
\)  we have 
\begin{equation}\label{eq rho coeff}
\rho_i=\nu(\alpha)^{i-e} \sum_{j=i}^n\binom{j}{i}\varphi_j \alpha^{j-i}.
\end{equation}
Denote the roots of \(\varphi\) by \(\alpha=\alpha_1,\dots,\alpha_n\in\overline{K}\).
Then the roots of $\rho$ are $\frac{\alpha_i-\alpha}{\nu(\alpha)}$ for $1\le i\le n$
and thus the slopes of the ramification polygon \(\rampol(\varphi)\) 
determine the valuations of the differences of the roots of $\varphi$.
%
%
%



\begin{lemma} \label{lem shape}
Let \(\varphi\in\OK[x]\) be a \(\nu\)-{\O}ystein polynomial of degree \(n=fe=f\epsilon p^w\)
and set \(\KL=\KK[x]/(\varphi)\).
Then the ramification polygon \(\rampol(\varphi)\) of \(\varphi\) 
is obtained from the ramification polygon \(\rampol_{\KL/\Lur}\) of
the totally ramified extension \(\KL/\Lur\) by adding a segment with endpoints
\((-n,n-e)\), \((-e,0)\).
\end{lemma}

\begin{proof}
There are \(n-e\) roots \(\widetilde\alpha\) of \(\varphi(x)\) such that \(\widetilde\alpha\not\sim\alpha\); these
all give roots of \(\rho(x)\) with \(L\)-valuation \(-1\).  Let \(\psi(x)\) be the minimal
polynomial over \(\Lur\) of the uniformizer \(\piL=\nu(\alpha)\) of \(L\).
Then the roots of \(\psi\) have the form \(\widetilde\pi_L=\nu(\widetilde\alpha)\),
where \(\widetilde\alpha\) is a root of \(\varphi\) such that
\(\widetilde\alpha\sim\alpha\).  Since \(\nu(x)\in\OK[x]\) and 
\(\overline\nu(x)\) is separable, we have \(v_\KL(\widetilde\pi_L-\piL)=v_L(\widetilde\alpha-\alpha)\).  
Hence the multisets
\(\{v_\KL(\frac{\widetilde\alpha-\alpha}{\nu(\alpha)}):\varphi(\widetilde\alpha)=0,\;\widetilde\alpha\sim\alpha\}\) 
and \(\{v_\KL(\frac{\widetilde\pi_L-\piL}{\piL}):\psi(\widetilde\pi_L)=0\}\) are equal.  
The lemma now follows from the relation
between the slopes of the Newton polygon of \(\rho(x)\) and the valuations of its roots. 
\end{proof}

It follows from the lemma that \(\rampol(\varphi)\) is determined by the extension \(\KL/\KK\), so
we may denote it by \(\rampol_{\KL/\KK}\).  Using the description of the ramification
polygon for the totally ramified case (see \cite[Lemma 1]{scherk} or \cite[Lemma 3.2]{pauli-sinclair})
we obtain the ramification polygon of a general extension \(\KL/\KK\) of degree \(fe=f\epsilon p^w\) 
\begin{align}\label{eq ram pol}
\rampol_{\KL/\KK}=\Bigl\{\bigl(-ef,(f-1)e\bigr),(-e,0),&
\bigl(-p^w,0\bigr),
\bigl(-p^{w_{1}},v_\KL(\rho_{w_{1}})\bigr),
\dots\nonumber\\[-1ex]
&\dots,\bigl(-p^{w_{u-1}},v_\KL(\rho_{w_{u-1}})\bigr),
\bigl(-1,v_\KL(\rho_1)\bigr)\Bigr\}
\end{align}
for some \(w=w_0>w_{1}>\dots>w_{u-1}>w_u=0\).
\arxivonly{
See Figures \ref{fig ram pol} and \ref{fig ram pol examples}.
}
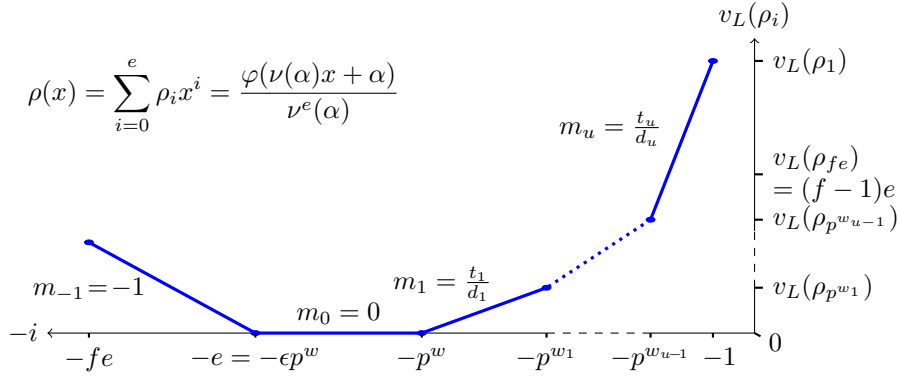
\begin{figure}
\begin{tikzpicture}[xscale=0.55,yscale=0.3]
\draw (1,0) -- (-1.5,0);
\draw[dashed] (-1.5,0) -- (-4,0);
\draw[->] (-4,0) -- (-16,0) node[anchor=east] {$-i$};

\draw (1,0) -- (1,2);
\draw[dashed] (1,2) -- (1,4.5);
\draw[->] (1,4.5) -- (1,13) node[anchor=south] {$v_\KL(\rho_i)$};

\draw[very thick,blue] (0,12) -- (-1.5,5);
\draw[very thick,blue,dotted] (-1.5,5) -- (-4,2);
\draw[very thick,blue] (-4,2) -- (-7,0) -- (-11,0);
\draw[very thick,blue] (-15,4) -- (-11,0) ;

\filldraw[blue] (0,12) circle (3pt)
                 (-1.5,5) circle (3pt)
                 (-4,2) circle (3pt)
                 (-7,0) circle (3pt)
                 (-11,0) circle (3pt)
                 (-15,4) circle (3pt);

\foreach \x/\xtext in {0/1\!\!\!, 1.5/p^{w_{u\!-\!1}}, 4/p^{w_{1}}, 7/p^{w},  11/e=-\epsilon p^w, 15/f e}
  \draw[thick] (-\x,0) -- (-\x,-0.2) node[anchor=north] {$-\xtext$};

\draw[thick] (1,0) -- (1.142,-0.142);
\draw (1.5,-0.4) node {$0$};

\foreach \x/\xtext in { 7/{\!\!\!\begin{array}{l}v_\KL(\rho_{f e})\\=(f-1)e\end{array}},2/v_\KL(\rho_{p^{w_{1}}}), 5/v_\KL(\rho_{p^{w_{u-1}}}), 12/v_\KL(\rho_1)}
  \draw[thick] (1,\x) -- (1.2,\x) node[anchor=west] {$\xtext$};

\draw (-12,10.5) node {$\rho(x)=\displaystyle\sum_{i=0}^{e} \rho_i x^i = \displaystyle\frac{\varphi(\nu(\alpha) x+\alpha)}{\nu^{e}(\alpha)}$};

\draw (-1,9)  node[anchor=east] {$m_{u}=\frac{t_{u}}{d_{u}}$}
      (-6.5,1)   node[anchor=south] {$m_1=\frac{t_1}{d_1}$}
      (-9,0)   node[anchor=south] {$m_0=0$}
      (-15,1)   node[anchor=south] {$m_{-1}\!=\!-1$};
\end{tikzpicture}
\caption{General form of the ramification polygon of a \(\nu\)-{\O}ystein polynomial 
$\varphi=\nu(x)^e+\sum_{i=0}^{e-1}\varphi^*_i(x)\nu(x)^i\in\OK[x]$ of degree \(n=f\cdot e=f\cdot \epsilon p^w\).  The segment of slope \(-1\) only exists when \(f>1\) and a segment of slope 0 indicates 
$\epsilon>1$.
}\label{fig ram pol}
\end{figure}


We have
\(
v_\KK(\disc(\KL/\KK))=v_\KK(\disc(\varphi))=f\cdot (e+v_\KL(\rho_1)-1)
\),
and each of the segments of \(\rampol\) corresponds to a subextension of \(\KL/\KK\) 
(see \cite[Algorithmus 4.4]{greve} or \cite[Section 3.4]{milstead-phd} for an algorithm).




\arxivonly{
\begin{figure}
\begin{tikzpicture}[scale=.26]
\draw (-10,14) node {$f=1$, $e=18$};

\draw[->] (0,0) -- (0,16);
\draw[->] (0,0) -- (-19,0);

\draw[very thick,green] (-18,0) -- (-9,0) -- node[anchor=east] {$\frac{13}{8}$} (-1,13);
                 
\filldraw[red] (-3,3) circle (5pt);
\draw[very thick, red,dotted] (-18,0) -- (-9,0) -- node[anchor=west] {$\frac{1}{2}$} (-3,3) -- node[anchor=west] {$5$} (-1,13);

\filldraw[blue] (-3,6) circle (5pt);
\draw[very thick,blue,dashed] (-18,0) -- node[anchor=south] {\color{black}$0$}  (-9,0) -- node[anchor=west] {$1$} (-3,6) -- node[anchor=east] {$\frac{7}{2}$}(-1,13);

\filldraw[black] (-18,0) circle (5pt) (-9,0) circle (5pt) (-1,13) circle (5pt);

\foreach \x in {-18,-9,-3,-1}  
\draw[thick] (\x,0) -- (\x,-0.2) node[anchor=north] {$\x$};

\draw[thick] (0,0) -- (0.142,-0.142);
\draw (0.7,-0.6) node {$0$};

\foreach \x in {13,6,3} 
  \draw[thick] (0,\x) -- (0.2,\x) node[anchor=west] {$\x$};

\end{tikzpicture}
\begin{tikzpicture}[scale=0.26]
\draw (-10,14) node {$f=2$, $e=9$};

\draw[->] (0,0) -- (0,16);
\draw[->] (0,0) -- (-19,0);

\draw[very thick,green] (-18,9) --  (-9,0) -- node[anchor=south] {$\frac{7}{8}$}(-1,7);

\filldraw[blue] (-3,3) circle (5pt) ;
\draw[very thick,blue,dashed] (-18,9) -- node[anchor=west,black] {$-1$}  (-9,0) -- node[anchor=west] {$\frac{1}{2}$}(-3,3) -- node[anchor=west] {$2$}(-1,7);

\filldraw[black] (-18,9) circle (5pt) (-9,0) circle (5pt) (-1,7) circle (5pt);
                 
\foreach \x in {-18,-9,-3,-1}  
\draw[thick] (\x,0) -- (\x,-0.2) node[anchor=north] {$\x$};

\draw[thick] (0,0) -- (0.142,-0.142);
\draw (0.7,-0.6) node {$0$};

\foreach \x in {9,7,3} 
  \draw[thick] (0,\x) -- (0.2,\x) node[anchor=west] {$\x$};
\end{tikzpicture}

\begin{tikzpicture}[scale=0.26]
\draw[->] (0,0) -- (0,18);
\draw[->] (0,0) -- (-19,0);
\draw (-10,14) node {$f=3$, $e=6$};
\draw[very thick,green] (-18,12) -- node[anchor=west,black] {$-1$} (-6,0) -- node[anchor=south] {$0$}  (-3,0) -- node[anchor=east] {$\frac{5}{2}$} (-1,5);
\filldraw[green] (-18,12) circle (5pt) (-6,0) circle (5pt) (-3,0) circle (5pt) (-1,5) circle (5pt);

\foreach \x in {-18,-6,-3,-1}  
\draw[thick] (\x,0) -- (\x,-0.2) node[anchor=north] {$\x$};

\draw[thick] (0,0) -- (0.142,-0.142);
\draw (0.7,-0.6) node {$0$};

\foreach \x in {12,5} 
  \draw[thick] (0,\x) -- (0.2,\x) node[anchor=west] {$\x$};

\end{tikzpicture}
\begin{tikzpicture}[scale=0.26]

\draw (-10,14) node {$f=6$, $e=3$};
\draw[->] (0,0) -- (0,18);
\draw[->] (0,0) -- (-19,0);
\draw[very thick,blue, dashed] (-18,15) -- node[anchor=west,black] {$-1$} (-3,0) -- node[anchor=west] {$\frac{3}{2}$}  (-1,3);
\filldraw[very thick,blue]  (-18,15) circle (5pt) (-3,0) circle (5pt) (-1,3) circle (5pt);
\foreach \x in {-18,-3,-1}  
\draw[thick] (\x,0) -- (\x,-0.2) node[anchor=north] {$\x$};

\draw[thick] (0,0) -- (0.142,-0.142);
\draw (0.7,-0.6) node {$0$};

\foreach \x in {15,3} 
  \draw[thick] (0,\x) -- (0.2,\x) node[anchor=west] {$\x$};
\end{tikzpicture}







\caption{Ramification polygons of extensions $\KL$ of degree 18 over 
\(\Q_3\) with \(v_3(\disc(\KL/\Q_3))=30\), also see \href{https://www.lmfdb.org/padicField/?p=3&n=18&c=30&search_type=Families}{the list of these families in the LMFDB}.}
\label{fig ram pol examples}
\end{figure}
}

\subsection*{Hasse-Herbrand function}\label{sec herbrand}
Let \(\varphi\in\OK[x]\) be Eisenstein of degree \(e\) and let \(\alpha\) be a root of \(\varphi\).
Set \(\KL=\KK(\alpha)\) and let 
\(\rho(x)=\frac{\varphi(\alpha x+\alpha)}{\alpha^e}\)
be the ramification polynomial of \(\varphi\).
Denote the endpoints of the segments with positive slope of the ramification polygon 
of \(\varphi\) by \((-p^{w_i},v_\KL(\rho_{p^{w_i}}))\)
where \(0=w_u<\dots<w_0=w\).  Let \(\Phi_{\KL/\KK}:[0,\infty)\rightarrow[0,\infty)\)
be the Hasse-Herbrand function of \(\KL/\KK\), which is defined in \cite[IV\,\S3]{Se} and
discussed in \cite[Section~2.2]{families}.  Then  we have (compare
\cite[Section 2]{Mo} and \cite[Section 3]{pauli-sinclair})
\begin{equation}\label{eq hasse-herbrand}
e\Phi_{\KL/\KK}(m)=\!\min_{0\leq i\leq u}\!\{{v_\KL(\rho_{p^{w_i}}) + m p^{w_i}}\!\}
=\!\min_{1\le i\le e}\! \{{v_\KL(\rho_i)+i\!\cdot\!m}\}=
\cont{\piL}{\rho(\piL^m x)}
\end{equation}
for \(m\ge0\).  Now suppose, \(\varphi\in\OK[x]\) is \(\nu\)-{\O}ystein of degree \(fe\),
\(\Lur=\KK[x]/(\nu)\), \(\KL=\KK[x]/(\varphi)\) with $\nu$-ramification polynomial \(\rho\) and ramification polygon \(\rampol\).
Then \(\Phi_{\Lur/\KK}(x)=x\) and thus
\(\Phi_{\KL/\KK}=\Phi_{\Lur/\KK}\circ\Phi_{\KL/\Lur}=\Phi_{\KL/\Lur}\).
Using (\ref{eq hasse-herbrand}) and Lemma \ref{lem shape} we get
\begin{equation}\label{eq hasse-herbrand 1}
e\Phi_{\KL/\KK}(m)=\min_{1\le i\le n} \{ v_\KL(\rho_i)+i\cdot m \}
=\cont{\nu(\alpha)}{\rho\left(\nu(\alpha)^m x\right)}.
\end{equation}
Because \(\Phi_{\KL/\KK}\) only depends on the ramification polygon \(\rampol\) 
of \(\KL/\KK\), we can set \(\Phi_\rampol=\Phi_{\KL/\KK}\).

\section{Residual polynomials}
\label{sec residual polynomials}

Let \(\rampol\) be the Newton polygon of a polynomial \(\rho\in\OL[x]\).   
To each segment \(\mathcal{S}\) of \(\rampol\) we attach a polynomial 
\(\overline{A}\in\RL[x]\).  These polynomials were first introduced as 
\emph{zugeordnete} (associated) polynomials by Ore 
\cite{ore-newton,montes-nart}
and have been called residual polynomials in more recent work.
The residual polynomials of the segments of the Newton polygon of a polynomial yield information about the unramified part of the extensions generated by the factors of the polynomial.  

Let $\rho(x)=\sum_{i=0}^{n} \rho_i x^i\in\OL[x]$
with roots \(\beta_1,\dots,\beta_n\)
and denote by \(\rampol\) its Newton polygon.
Let \((-i,v_\KL(\rho_i))\) and \((-j,v_\KL(\rho_j))\) be the endpoints of a segment of \(\rampol\)
of slope \(m=\frac{t}{d}\)
and let \(\beta\in\overline{L}\) be a root of \(\rho\) with \(v_\KL(\beta)=m\).
We divide the polynomial \(\rho(\beta x)\) whose roots are
\(\frac{\beta_1}{\beta},\dots,\frac{\beta_n}{\beta}\)
by its content \(\piL^{v_\KL(\rho_j)}\beta^j\) and get
\begin{align*}
\frac{\rho(\beta x)}{\piL^{v_\KL(\rho_j)}\beta^j}
=\sum_{l=0}^n \frac{\rho_l\beta^l x^l}{\piL^{v_\KL(\rho_j)}\beta^j}
\equiv\sum_{l=j}^i \frac{\rho_l\beta^{l-j}}{\piL^{v_\KL(\rho_j)}} x^l
\equiv\!\!\!\sum_{k=0}^{(i-j)/d} \frac{\rho_{j+kd}\beta^{kd}}{\piL^{v_\KL(\rho_j)}} x^{j+kd}
\bmod \piL.
\end{align*}
Let \(\theta=\frac{\beta^d}{\piL^t}\).
Dividing the above by \(x^j\), replacing \(\beta^d\) with \(\theta\piL^t\)
and replacing \(\theta\cdot x^d\) with \(z\) yields the residual polynomial of the segment of
slope \(m\) of \(\rampol\):
\begin{definition}\label{def res}
Let $\rho(x)=\sum_{i=0}^{n} \rho_i x^i\in\OL[x]$ and suppose the Newton polygon
of \(\rho\) has a segment \(\mathcal{S}\) with endpoints 
\(\left(-i,v_\KL(\rho_{i})\right)\)
and 
\(\left(-j,v_\KL(\rho_j)\right)\) 
of slope \(m=\frac{t}{d}\).  Then the 
\emph{residual polynomial} of \(\mathcal{S}\) is
\begin{equation*}
\overline{A}_m(z):=
\sum_{k=0}^{(i-j)/d} \overline{\rho_{j+kd}\piL^{kt-v_\KL(\rho_j)}} z^{k}\in\RL[z].
\end{equation*}
\end{definition}
Each root of \(\overline{A}_m\) has the form \(\overline{{\theta}\cdot\left(\frac{\beta_i}{\beta}\right)^d}
=\overline{\left(\frac{\beta_i^d}{\piL^t}\right)}\), where \(\beta_i\) is a root of \(\rho\) with
\(v_\KL(\beta_i)=m=\frac{t}{d}\).

Denote the slopes of \(\rampol\) by \(m_1<m_2<\dots<m_u\) and the residual polynomials by \(\overline{A}_{m_i}\).  It follows directly from Definition \ref{def res} that the constant coefficient of 
\(\overline{A}_{m_i}\) is equal to the leading coefficient of \(\overline{A}_{m_{i+1}}\).
The nonzero terms of $\overline{A}_m$ correspond to the points of the form $(-h,v_L(\rho_h))$ on the segment of slope $m$.

\subsection*{The residual invariant}
We now consider the residual polynomials attached to the segments of the ramification polygon
$\rampol(\varphi)$ of an Eisenstein  polynomial \(\varphi\).
If \(m_1,\dots,m_u\) are the slopes of \(\rampol(\varphi)\) and 
\(\overline{A}_{m_1}\) 
is the residual polynomial of the segment of slope \(m_i\)
we write
\(\Ainv_{\varphi}=(\overline{A}_{m_1},\dots,\overline{A}_{m_u})\).

Let \(\alpha=\alpha_1,\dots,\alpha_e\in\Ksep\) be the roots of \(\varphi\).  Then
the roots of the ramification polynomial \(\rho(x)=\frac{\varphi(\alpha x+\alpha)}{\alpha^e}\) are of the form \(\frac{\alpha_i-\alpha}{\alpha}\),
and if \(v_L\left(\frac{\alpha_i-\alpha}{\alpha}\right)=m=\frac{t}{d}\) we have 
\(\overline{A}_m\left(\frac{(\alpha_i-\alpha)^d}{\alpha^{d+t}}\right)=0\).

Let \(\delta\in\OL^\times\) and \(\beta=\delta\alpha\).
Then the minimal polynomial \(\psi\in\OK[x]\) of \(\beta\) is Eisenstein and has the same ramification polygon \(\rampol\) as 
\(\varphi\).
Let \(\overline{\left(\frac{(-1+\alpha_i/\alpha)^d}{\alpha^t}\right)}\)
be a root of the residual polynomial \(\Apol_m\) of the ramification polynomial \(\rho\) of 
\(\varphi\) for the segment of \(\rampol\) with slope \(m=\frac{t}{d}\).  Then by
the proof of \cite[Proposition 4.4]{greve-pauli}, 
the corresponding residual polynomial \(\overline{B}_m(z)\) of the ramification polynomial \(\sigma\) of \(\psi\) has a root
\(\overline{\left(\frac{(-1+\beta_i/\beta)^d}{\beta^t}\right)}\)
with
\[
\frac{(-1+\beta_i/\beta)^d}{\beta^t}
\sim
\frac{1}{\delta^t}
\frac{(-1+\alpha_i/\alpha)^d}{\alpha^t}.
\]
Thus \(\overline{B}_m(z)=\overline{\gamma}\overline{A}_m(\overline{\delta}^{t}z)\) for some \(\overline\gamma\in\RK^{\times}\).  

Since the ramification polynomials \(\rho\) and \(\sigma\) are monic 
the respective residual polynomials \(\overline{A}_{m_1}\) and \(\overline{B}_{m_1}\) of the first segment (with slope 
\(m_1=\frac{t_1}{d_1}\))
are also monic.  Therefore we get
\[
\overline{B}_{m_1}(z)=\overline\delta^{(-t_1)\deg A_{m_1}}\overline{A}_{m_1}(\overline{\delta}^{t_1}z).
\]
Because all uniformizers of \(\KL/\KK\) are of the form \(\beta=\delta\alpha\) where \(v_\KL(\delta)=0\) and because the residual
polynomials are only affected by \(\overline{\delta}\in\RK^{\times}\) as described above we obtain an invariant of \(\KL/\KK\). 
Note that our formulation of the invariant \(\Ainv\) differs in the ordering of the polynomials from \cite[Theorem 4.7]{pauli-sinclair}, because of the different ordering of the segments in the ramification polygon.

\begin{theorem}\label{theo A}
Let $\varphi\in\OK[x]$ be an Eisenstein polynomial with ramification polygon $\rampol$.
Denote by $m_{1}=\frac{t_1}{d_1},\dots,m_{u}=\frac{t_u}{d_u}$ the slopes of the segments of
$\rampol$ and let $\overline{A}_{m_i}$ be the residual polynomial of the segment of slope \(m_i\).
Let \(L=K[x]/(\varphi)\).
Then the following is an invariant of \(\KL/\KK\):
\begin{equation}\label{eq A}
\Ainv_{\KL/\KK}=\left\{\left(
\overline\gamma_1 \overline{A}_{m_1}(\overline{\delta}^{t_1}z),
\overline\gamma_2 \overline{A}_{m_2}(\overline{\delta}^{t_2}z),
\dots,
\overline\gamma_u \overline{A}_{m_u}(\overline{\delta}^{t_u}z)
\right) :\overline\delta\in\RK^\times
\right\}
\end{equation}
where
\(\overline\gamma_1=\overline\delta^{(-t_1)\deg A_{m_1}}\) and
\(\overline\gamma_i=\overline\gamma_{i-1}\displaystyle\overline\delta^{(-t_i)\deg A_{m_i}}\) for \(i=2,\dots,u\). 
We write \[\Ainv_{\KL/\KK}=[\overline{A}_{m_1},\dots,\overline{A}_{m_u}].\]
\end{theorem}

We say that \(\Ainv_{\KL/\KK}\) is the {\em residual invariant} of the extension \(\KL/\KK\).
This invariant is an important tool in the enumeration of generating polynomials for extensions
of local fields \cite{pauli-sinclair}.
Together with the ramification polygon and the constant coefficient modulo \(\piK^2\)
the residual invariant determines a subextension $T$ of the splitting field $N$ of 
an Eisenstein polynomial such that $N/T$ is a $p$-extension
\cite[Theorem 9.1]{greve-pauli}.
This data also determines the Galois groups of Eisenstein 
polynomials whose ramification polygon consists of a single segment \cite[Section 8]{greve-pauli}.

\subsection*{The residual invariant for general extensions}

It is natural to extend the definition of the residual invariant to include extensions
\(\KL/\KK\) generated by a root \(\alpha\) of an {\O}ystein polynomial
\(\varphi(x)=\nu(x)^e+\sum_{i=0}^{e-1}\varphi^*_i(x)\nu(x)^i\) by setting
\[
\overline{B}_m(z)=\sum_{k=0}^{(i-j)/d}\overline{\rho_{j+kd} \nu(\alpha)^{kt-v_L(\rho_j)}}z^{k}\in\RL[z].
\]
Unfortunately, in this setting there may exist $\delta_1,\delta_2 \in \OK^\times$ with 
$\delta_1 \equiv \delta_2 \pmod{\pi_K}$ but
\(\overline\gamma \overline{B}_m(\overline\delta_1^t)\ne \overline\gamma \overline{B}_m(\overline\delta_2^t)\)
for some slope \(m=\frac{d}{t}\), which means that the equivalence relation used to define 
\(\Ainv_{\KL/\KK}\) in the Eisenstein case may not be well-defined.
To avoid this difficulty we base the residual invariant of \(\KL/\Lur/\KK\) on \(\Ainv_{\KL/\Lur}\).
Therefore we need to consider the action of \(\Gal(\Lur/\KK)\cong\Gal(\RL/\RK)\) on \(\Ainv_{\KL/\Lur}\):

\begin{definition}\label{def res poly class}
Let \(\KL/\Lur/\KK\) and let 
\((\overline{A}_{m_1},\dots,\overline{A}_{m_u})\) be a representative of the residual invariant \(\Ainv_{\KL/\Lur}\) of the totally ramified extension 
\(\KL/\Lur\).
Let \(\gamma_i\) ($1\le i\le u$) be as in Theorem \ref{theo A}.
We call
\[
\Ainv_{\KL/\KK}=\left\{\left(
\sigma\bigl(\overline\gamma_1 \overline{A}_{m_1}\!(\overline{\delta}^{t_1}\!z)\bigr),
\dots,
\sigma\bigl(\overline\gamma_u \overline{A}_{m_u}\!(\overline{\delta}^{t_u}\!z)\bigr)
\right) :\overline\delta\!\in\!\RL^\times\!\!,\;\sigma\!\in\!\Aut(\RLur/\RK)
\right\}
\]
the \emph{residual invariant} of \(\KL/\KK\). 
\end{definition}

\subsection*{Distinguished residual polynomials}
\label{sec dist res}

We choose a distinguished representative of each residual invariant \(\Ainv_{\KL/\KK}\)
defined in Theorem \ref{theo A} using the following ordering:

\begin{definition}\label{def dist res}
Let \(\langle \overline\xi\rangle =\Fq^\times \), \(g(z)=\sum_{i=0}^{n} g_i z^i \in \Fq[z]\) and
\(h(z)=\sum_{i=0}^n h_i z^i \in \Fq[z]\).
We write \(g >_{\overline\xi} h\) when \(g_i = h_i\) 
for \(i>j\) and \(\log_{\overline\xi} g_j > \log_{\overline\xi} h_j\) for some \(0\le j\le n\).
With lexicographic ordering, we extend this ordering to the tuples of polynomials in the residual invariant of an Eisenstein polynomial. We call the smallest tuple in the residual invariant distinguished.
\end{definition}

To find the distinguished representative of \(\Ainv\), we find \(\overline\delta\in\RK^{\times}\) such that
\begin{equation}\label{eq A tuple}
\left(
\overline\gamma_1 \overline{A}_{m_1}(\overline{\delta}^{t_1}z),
\overline\gamma_2 \overline{A}_{m_2}(\overline{\delta}^{t_2}z),
\dots,
\overline\gamma_u \overline{A}_{m_u}(\overline{\delta}^{t_u}z)
\right)
\end{equation}
is minimal with respect to \(>_{\xiRK}\)
where
\(\overline\gamma_1=\overline\delta^{(-t_1)\deg A_{m_1}}\) and
\(\overline\gamma_i=\overline\gamma_{i-1}\displaystyle\overline\delta^{(-t_i)\deg A_{m_i}}\) for \(i=2,\dots,u\). 
Repeated applications of Lemma \ref{lem min} below yield \(l=b+c\Delta\) such that 
(\ref{eq A tuple}) is minimal for \(\overline\delta=\xiRK^{l}\).

When \(a=bq+r\) for \(a,q\in\Z\), \(b\in\N\), and \(0\le r<b\) we write \(r=a\bmod b\)
and \(q=a\bdiv b\).  The following elementary lemma will be useful:

\begin{lemma}\label{lem min}
Let \(Q,e\in\N\) and \(a\in\N_0\).  
Set \(g=\gcd(e,Q)=se+tQ\), \(d=a\bdiv g\), \(\Delta=Q \bdiv g\), 
\(b=(-s\cdot d)\bmod Q\), and
\[
m= \min\left\{(a+k\cdot e) \bmod Q: k\in\Z\right\}.
\]
Then $m=a\bmod g$, and $m=(a+k\cdot e)\bmod Q$ if and only if
$k=b+c\Delta$ for some $c\in\Z$.
\end{lemma}

\begin{definition}\label{def dist res conj}
The \emph{distinguished residual polynomial tuple} of \(\KL/\Lur/\KK\) 
is the tuple that is distinguished among the conjugates of the residual invariant
of \(\KL/\Lur\) under the automorphisms of \(\Lur/\KK\).
\end{definition}

\arxivonly{
\begin{example}[\href{https://www.lmfdb.org/padicField/family/3.2.9.30b}{LMFDB 3.2.9.30b}]
Let \(\Q_3(\xi)\) be the unramified extension of \(\Q_3\) of degree \(2\).
We consider the extensions \(\KL/\Q_3(\xi)\) with ramification polygon 
\[
\rampol=\{(-9,0),(-3,3),(-1,7)\}
\]
Their residual polynomials are of the form:
\[
(z^3+\gamma, \gamma z^2+\delta) \text{ where }\gamma\in\F_9^\times, \delta\in \F_9^\times
\]
These \(64\) polynomials form \(8\) residual invariants which we represent by their distinguished representatives:
\[
\{ [z^3+1,z^2+\overline\xi^j] : 0\le j<8\}
\]
Over \(\Q_3\) we also consider conjugation by the nontrivial automorphism
\(\sigma(\overline{\xi})=\overline\xi^3\) of \(\F_3(\overline{\xi})/\F_3\).
Here we get \(5\) residual invariants:
\begin{align*}
&[z^3 + 1,z^2 + 1],\;[z^3 + 1,z^2 + 2],\; [z^3 + 1,z^2 + \overline\xi]=[z^3 + 1,z^2 + \overline\xi^3],\\
&[z^3 + 1,z^2 + \overline\xi^2]=[z^3 + 1,z^2 + \overline\xi^6],\;[z^3 + 1,z^2 + \overline\xi^5]=[z^3 + 1,z^2 + \overline\xi^7].
\end{align*}
Thus the five distinguished residual polynomial tuples of the extensions \(\KL/\Q_3\) are 
\((z^3 + 1,z^2 + 1)\), 
\((z^3 + 1,z^2 + 2)\), 
\((z^3 + 1,z^2 + \overline\xi)\), 
\((z^3 + 1,z^2 + \overline\xi^2)\), and 
\((z^3 + 1,z^2 + \overline\xi^5)\).
\end{example}
}

\subsection*{Families and subfamilies}
Let \(\varphi\in\OK[x]\) be Eisenstein and set \(\KL=\KK[x]/(\varphi)\).
Then the first \(\piK\)-adic digit of the constant coefficient
of all defining Eisenstein polynomials of \(\KL/\KK\) is in 
\(
\Cinv_{\KL/\KK}= 
\{\overline\varphi_{0,1}\delta^n:\delta\in\RK^\times\}\),
which is an invariant of \(\KL/\KK\).

For an {\O}ystein polynomial \(\varphi\in\OK[x]\) we can choose 
\(\varphi^*_{0,1}\) such \(\overline\psi_{0,1}=\overline\varphi^*_{0,1}(\overline\alpha)\) where \(\psi\in\Lur[x]\) is a defining Eisenstein polynomial of 
\(\KL=\KK[x]/(\varphi)\) over \(\Lur\).  
Since the \(\Aut(\Lur/\KK)\)-conjugates of \(\psi\) yield the same 
extension over \(\KK\), in this setting the invariant becomes
\(
\Cinv_{\KL/\KK}= 
\{\sigma(\overline\varphi_{0,1})\delta^n:\sigma \in \Aut(\RL/\RK),\,\delta\in\RL^\times\}
\).
To combine \(\Cinv_{\KL/\KK}\) with the residual invariant  
\(\Ainv_{\KL/\KK}\) we restrict the transformations of the residual 
polynomials to those that fix a representative of 
\(\Cinv_{\KL/\KK}\).  We set (compare \cite[(4.2)]{pauli-sinclair} for the totally ramified case):
\[
\Ainv_{\KL/\KK}^*=\left\{\left(
\sigma\bigl(\overline\gamma_1 \overline{A}_{m_1}\!(\overline{\delta}^{t_1}\!z)\bigr),
\dots,
\sigma\bigl(\overline\gamma_u \overline{A}_{m_u}\!(\overline{\delta}^{t_u}\!z)\bigr)
\right) :\overline\delta\!\in\!\RL^\times\!\!,\,\delta^n\!\!=\!1,\,\sigma\!\in\!\Aut(\RLur/\RK)
\right\}
\]
with \(\gamma_i\) as in Theorem \ref{theo A}.
With the above we refine the definition of family from
\cite{Mo} and \cite[Section 1.2]{families} and obtain:


\begin{definition}
A \emph{family} \(\mathcal{F}\) of extensions of \(\KK\) is the set of all extensions
of \(\KK\) that have the same ramification polygon.
We call the subset of all extensions \(\KL/\KK\) in \(\mathcal{F}\) that have the same 
\(\Cinv_{\KL/\KK}\) and the same \(\Ainv_{\KL/\KK}^*\) a \emph{subfamily}.
\end{definition}

This notion of subfamily is used in defining labels for $p$-adic fields in the LMFDB, but does not reappear in the remainder of this article.

\section{Reduction} \label{sec reduction}

Reduced Eisenstein polynomials were introduced by Krasner \cite{krasner-cluj} and their theory was brought into modern form by Monge \cite{Mo}.  We recall their construction and generalize it to {\O}ystein polynomials.  

\subsection*{Zeroth reduction step for Eisenstein polynomials} \label{sec first eisenstein}
We refine the first reduction step \cite[Algorithm 1]{Mo} for an Eisenstein polynomial
\(\varphi\in\OK[x]\), in which the
leading coefficient \(\varphi_{0,1}\) of the constant term \(\varphi_0\) is reduced.
Because transforming \(\varphi_{0,1}\) also impacts the residual polynomials,
out of all the transformations that lead to a reduced value for \(\varphi_{0,1}\),
we choose those that yield the smallest residual polynomials with respect to \(>_{\xiRK}\)
(see Definition \ref{def dist res}).

Fix \(\xiRK\in\RK\) such that \(\langle\xiRK\rangle=\RK^\times\) and let \(\xiK\in\OK\) be a lift of \(\xiRK\).
We first find all \(s\in\N_0\) such that the transformation \(x\mapsto x\xiK^{-s}\)
yields polynomials
\[
\psi(x)=\xiK^{es}\varphi(x \xiK^{-s})=x^e+\dots+\varphi_{0,1}\xiK^{es}
\]
with \(\log_{\xiRK}(\overline\varphi_{0,1}\xiRK^{es})\) minimal.
That is, with Lemma \ref{lem min} we find all \(s\) for which 
\[(a+s\cdot e) \bmod (q-1)= \min\left\{(a+k\cdot e) \bmod (q-1): 0\le k \le q-2\right\}
\] 
where \(a=\log_{\xiRK}\overline\varphi_{0,1}\).  Among these polynomials we find those whose residual polynomials are minimal given the 
constraints above.  This choice assures compatibility with the generic defining polynomials \cite{families} but also means that the residual polynomials of distinguished polynomials in general do not match their distinguished representatives.  

The definition of being a {\em distinguished} $\nu$-{\O}ystein polynomial is given in Definition~\ref{con main} below.  We build up the definition through a series of conditions.

\begin{condition}\label{con dist const}
For an Eisenstein polynomial \(\varphi\in\OK[x]\) to be distinguished we require that
for all Eisenstein polynomials \(\psi\) with 
\(\KK[x]/(\psi)\cong\KK[x]/(\varphi)\) we have 
\begin{enumerate}
\item \(\log_{\xiRK}(\overline{\varphi_{0,1}})\le \log_{\xiRK}(\overline{\psi_{0,1}})\) and 
\item \(\Ainv_\varphi\le\Ainv_\psi\) when \(\log_{\xiRK}(\overline{\varphi_{0,1}})=\log_{\xiRK}(\overline{\psi_{0,1}})\). 
%
\end{enumerate}
\end{condition}
When \(e=p^w\)  exactly one representative of \(\mathcal{A}_{\KL/\KK}\) satisfies Condition
\ref{con dist const}.
Reducing \(\varphi_{0,1}\) is the first step in our reduction algorithm for Eisenstein polynomials.

\Algo{polred0ram}{alg polred0ram}
{
Eisenstein \(\varphi\in\OK[x]\) of degree \(e\) with ramification polynomial \(\rho\)
}
{
\(\left\{ \psi(x)=\delta^e\varphi(\frac{1}{\delta x}):\delta\in\RepK \text{ and } \psi \text{ satisfies Condition \ref{con dist const}} \right\}\) 
and a distinguished representative of \(\Ainv_{\KL/\KK}\) 
}
{
Write \(\langle\xiRK\rangle=\RK^\times\).\\[-1ex]
\begin{enumerate}
\item compute the residual polynomials \((A_{m_1},\dots A_{m_u})\) of \(\rho\).
\item find the set \(E\) of all \(\delta\in \RK^\times\) such that 
\(
\log_{\xiRK}(\varphi_{0,1}\delta^e)\bmod (\card\RK-1)
\) is minimal.
\item find the set \(D\) of all \(\delta\in E\) such that
\[
(B_{m_1},\dots,B_{m_u})=
\left(
\overline\gamma_1 \overline{A}_{m_1}(\overline{\delta}^{t_1}z),
\dots,
\overline\gamma_u \overline{A}_{m_u}(\overline{\delta}^{t_u}z)
\right)
\] is minimal with respect to \(\ge_{\xiRK}\) (see Theorem \ref{theo A} and Definition \ref{def dist res}).
\item return  \(\left\{ \psi(x)=\delta^e\varphi(\frac{1}{\delta}x):\delta\in D\right\}\)
and \((B_{m_1},\dots,B_{m_u})\).
\end{enumerate}
}
\arxivonly{
\begin{example}
\label{ex polred0ram}
Let \(K=\Q_3(\xi)\) where \(\nu(\xi)=0\) for \(\nu(x)=x^3 + 2x + 1\)
and \[\varphi(x)=x^6 + 3\xi^{4}x^5 + 9\xi^{19}x + 3\xi^3\in\OK[x].\]
We apply the zeroth reduction step Algorithm \ref{alg polred0ram} to \(\varphi\).

The ramification polygon of \(\varphi\) is \((-6,0),(-3,0),(-1,5)\) with slopes \(0,\frac{5}{2}\).
and residual polynomials \((z^3 + 2, 2z + \overline\xi)\).
Applying Lemma \ref{lem min} twice we obtain
the distinguished representative
\[
((z^3+2),\overline\delta^{(-5)}\bigl(2\overline\delta^5 z+\overline\xi\bigl))=
(z^3 + 2,2z + 1)\quad \text{where }\overline\delta=\overline\xi^{21}
\]
of the residual invariant \([z^3 + 2,2z + \overline\xi]\).

We get minimal \(\log_{\overline\xi}\overline\xi^3\gamma^e=1\) for \(\gamma\in\{\overline\xi^4,\overline\xi^{17}\}\).
Thus the zeroth reduction step yields 
\begin{align*}
\psi_{1}(x)&=\xi^{6\cdot 4}\varphi(\xi^{-4}x)=
x^6 + 3\xi^{14} x^5 + 9\xi^{39}x + 3\xi^{27} \\
\psi_{2}(x)&=\xi^{6\cdot 17}\varphi(\xi^{-17}x)=
x^6 + 3\xi^{21} x^5 + 9\xi^{104}x + 3\xi^{105}
\end{align*}
The residual polynomials of \(\psi_1\) and \(\psi_2\)
are 
\(B_1=(z^3+2,2z+\overline\xi^{18})\) \and \(B_2=(z^3+2,2z+\overline\xi^7)\) respectively.
Since \(B_1\ge_\xi B_2\) the zeroth reduction step returns \(\psi_2\).
\end{example}
}

\subsection*{Constant coefficient of {\O}ystein polynomials}

Because for {\O}ystein polynomials we do not have a straightforward connection between constant coefficient and residual polynomial transformations, the following is simpler than Condition \ref{con dist const}.

Let $\varphi(x)=\nu(x)^{e}+\sum_{i=0}^{e-1}\varphi^*_i(x) \nu(x)^i \in \OK[x]$ be \(\nu\)-{\O}ystein.
In the following we will be concerned with the reduction of the \(\varphi^*_i\).
We write
\(\varphi_i^*(x)=\sum_{k=0}^\infty \varphi^*_{i,k}{\piK^k}\) 
where $\varphi^*_{i,k}=\sum_{j=0}^{f-1}\varphi^*_{i,k,j}x^j$ with \(\varphi^*_{i,k,j}\in\RepK\).

\begin{condition}\label{con dist const oystein}
Let  \(\varphi\in\OK[x]\) be \(\nu\)-{\O}ystein.  For \(\varphi\) to be distinguished we require that
for all \(\nu\)-{\O}ystein polynomials \(\psi\in\OK[x]\) with 
\(\KK[x]/(\psi)\cong\KK[x]/(\varphi)\) we have
\[
\log_{\xiRL}(\overline{\varphi}^*_{0,1}(\xiRL))\le \log_{\xiRL}(\overline\psi^*_{0,1}(\xiRL)).
\]
\end{condition}

\subsection*{Additive residual polynomials}
In \cite{pauli-sinclair} additive residual polynomials are called residual polynomials of components; in \cite{Mo} they are denoted by \(S_m\); and in \cite{krasner-cluj} they are denoted by \(U_m\).

\begin{definition}\label{def A add}
Let $\varphi(x)=\nu(x)^{e}+\sum_{i=0}^{e-1}\varphi^*_i(x) \nu(x)^i \in \OK[x]$ be \(\nu\)-{\O}ystein, $\alpha$ a root of $\varphi$, and $\rho(x)=\varphi(\nu(\alpha)x+\alpha)/\nu(\alpha)^e$ the ramification polynomial of $\varphi$.
For \(m\in\N_0\) we call
\[
\Aadd_m(x)
=\overline{\left(\frac{\rho(\nu(\alpha)^m x)}{\cont{\nu(\alpha)}{\rho(\nu(\alpha)^m x)}}\right)}
\]
the \textit{additive residual polynomial} of \(\varphi\) attached to the slope \(m\).
\end{definition}

\begin{lemma} \label{lem additive}
For \(m\in\N\), \(\Aadd_m(x)\) is additive.
Suppose the ramification polygon \(\rampol\) of \(\varphi\) has a segment of slope \(m\)
whose endpoints have abscissas \(-p^u,-p^v\), with \(0\le v<u\le w\).  Then
\(\Aadd_m(x)\) has top degree \(p^u\) and bottom degree \(p^v\).
\end{lemma}

\begin{proof}
Set \(\pi_L=\nu(\alpha)\), \(\varphi_m(x)=\varphi(\pi_L^{m+1}x+\alpha)\), and 
\(\theta_m(x)=\frac{\varphi_m(x)}{\cont{\pi_L}{\varphi_m}}\).  Then \(\theta_m(x)\in\OL[x]\) is a
primitive polynomial whose image in \(\RL[x]\) is \(\Aadd_m(x)\).
The roots of \(\theta_m(x)\) have the form \(\frac{\widetilde\alpha-\alpha}{\pi_L^{m+1}}\), where
\(\widetilde\alpha\) is a root of \(\varphi(x)\).  It follows from Lemma~\ref{lem shape} 
and (\ref{eq ram pol}) that there
are \(n-p^u\) roots \(\widetilde\alpha\) of \(\varphi(x)\) such that 
\(v_L\left(\frac{\widetilde\alpha-\alpha}{\pi_L^{m+1}}\right)<0\).
These correspond to primitive factors of \(\theta_m(x)\) of the form 
\(-\frac{\pi_L^{m+1}}{\widetilde\alpha-\alpha}x+1\).  The
remaining roots \(\widetilde\alpha\) of \(\varphi(x)\) satisfy 
\(v_L\left(\frac{\widetilde\alpha-\alpha}{\pi_L^{m+1}}\right)\ge0\),
and give primitive factors of the form \(x-\frac{\widetilde\alpha-\alpha}{\pi_L^{m+1}}\).  Hence
there is \(\zeta\in\RL^{\times}\) such that \(\zeta^{-1}\Aadd_m(x)\) is
the product of the factors 
\(x-\overline{\left(\frac{\widetilde\alpha-\alpha}{\pi_L^{m+1}}\right)}\in\RL[x]\) such
that \(v_L\left(\frac{\widetilde\alpha-\alpha}{\pi_L^{m+1}}\right)\ge0\).

As in the proof of Lemma~\ref{lem shape} we let \(\psi(x)\) be the minimal
polynomial over \(\Lur\) of the uniformizer \(\piL=\nu(\alpha)\) for \(L\).  
Set \(\delta=\nu'(\alpha)\in\OL^{\times}\).  
Corresponding to each root \(\widetilde\alpha\) of \(\varphi(x)\) such that 
\(\widetilde\alpha\sim\alpha\) there is a root \(\widetilde\pi_L=\nu(\widetilde\alpha)\) of 
\(\psi(x)\) such that \(\widetilde\pi_L-\pi_L\sim\delta\cdot(\widetilde\alpha-\alpha)\). 
Hence there is \(\gamma\in\RL^{\times}\) such that
\(\gamma\Aadd_m(\delta^{-1}x)\) is the \(m\)th additive residual polynomial of
\(\psi(x)\).  Since the conclusions of the lemma hold for the Eisenstein polynomial 
\(\psi(x)\) \cite[Lemma~1]{scherk}, they hold for \(\varphi(x)\) as well.
\end{proof}

When \(\varphi\) is Eisenstein the residual polynomials of segments 
of the ramification polygon \(\rampol\) of \(\varphi\) with positive integral slopes correspond
to additive residual polynomials.  
Namely, when $\rampol$ has a segment \(\mathcal{S}\) of slope \(m\in\N\) then  
\(\Aadd_m(z) = z^j \Apol_m(z)\) where
$-j$ is the abscissa of the right endpoint of \(\mathcal{S}\).  
When $\rampol$ has no segment of slope \(m\) then \(\Aadd_m\) is a monomial.

Although the additive residual polynomials do not carry new information beyond the residual polynomials, they are
a useful tool for the reduction of Eisenstein and {\O}ystein polynomials as we will see in Proposition \ref{prop red} and Theorem \ref{theo red}.


As we have seen in (\ref{eq hasse-herbrand 1}), the content of \(\rho(\nu(\alpha)^m x)\)
is equal to \(e\Phi_\rampol(m)\), which only depends on the ramification polygon.

\subsection*{Reduction of Eisenstein polynomials}\label{sec red eisen}
We recall the method used in \cite{Mo} for reducing an Eisenstein polynomial \(\varphi\in\OK[x]\) and introduce our choice of representatives of the cokernel of the additive residual polynomials \(\Aadd\in\RK[x]\).
Let \(m\in\N\) and let \(\rho(x)=\frac{\varphi(\piL x+\piL)}{\piL^e}\) 
be the ramification polynomial of \(\varphi\). 
Because \(\varphi\) is Eisenstein the additive residual polynomial to the slope \(m\) is
\(\Aadd_m(x)=\overline{\frac{\rho(\piL^m x)}{\piL^{e\Phi_{\KL/\KK}(m)}}}\).
The reduction is based on the following result, which
describes how a change in the uniformizer of \(\KL\) changes
the coefficients of the defining polynomials in terms of
the additive residual polynomials, compare {\cite[pp.\,156--7]{krasner-cluj}},
{\cite[Lemma~2.5]{Mo}}, or {\cite[Proposition~5.5]{pauli-sinclair}}:

\begin{proposition}\label{prop red}
Let $\varphi\in\OK[x]$ 
be Eisenstein, let
$\piL$ be a root of $\varphi(x)$, and set $\KL=\KK(\piL)$.  Let $\gamma\in\OL^{\times}$
and $m\in\N$; then $\beta=\piL+\gamma\piL^{m+1}$ is
another uniformizer of $\KL$.  Let $\psi(x)\in\OK[x]$ be the minimal polynomial of $\beta$ over $\KK$.
Let \(\Aadd_m(z)\) be the additive residual polynomial of \(\varphi\) for the slope \(m\)
and set \(\eta=\piK/\piL^e\).  Then
\begin{enumerate}
\item \(\varphi_{i,k}=\psi_{i,k}\) for all $i,k$ such that $i+ek<e(\Phi_{L/K}(m)+1)$.  
\item Set \(h=e\Phi_{\KL/\KK}(m)\), \(i= h \bmod e\) and
\(k = (h\bdiv e)+1\).  Then
\[
{(\varphi_{i,k}-\psi_{i,k})}\eta^{k} \equiv\Aadd_m(\gamma)
\pmod\piL.
\]
\end{enumerate}
\end{proposition}

This proposition leads to a reduction algorithm (see \cite[Algorithm 2]{Mo}), in which starting at slope \(m=1\) the coefficients of \(\varphi\) are reduced iteratively.
If \(\Aadd_m:\RK\to\RK\) is surjective we can choose $\gamma\in\OL$ so that the transformation
\(\piL\mapsto\alpha+\gamma\piL^{m+1}\) gives \(\psi_{i,k}=0\) where $i,k$ are defined in statement (2).
Otherwise, we can still alter \(\overline{\varphi}_{i,k}\) 
by an element of the image of the additive polynomial \(\Aadd_m\):  


\begin{definition}\label{def red cokernel 2}
Let \(\overline\xi\) be a generator for \(\RL\) over \(\F_p\) and set
\(d=[\RL:\F_p]=f\cdot[\RK:\F_p]\); then \(B=\{1,\overline\xi,\dots,\overline\xi^{d-1}\}\)
is an \(\F_p\)-basis for \(\RL\).  Let \(\Aadd\in\RK[x]\) be additive and
let \(T\) be the \(d\times d\) matrix whose \(i\)th row is the coordinate vector of
\(\Aadd(\overline\xi^{i-1})\) with respect to \(B\).  
Let \(E\) be the reduced row echelon form of \(T\) and set \(r=\rank(E)=\rank(T)\).
Let \(\beta(x)\in\RK[x]\) and let \(\vec{v}\)
be the coordinate vector of \(\beta(\overline\xi)\) with respect to \(B\).
For \(1\le i\le r\) let \(E_i\) be the \(i\)th row of \(E\) and assume that the
leftmost nonzero entry of \(E_i\) is
in column \(j_i\).  Set \(\vec{w}=\vec{v}-\sum_{i=1}^rv_{j_i}E_i\), where \(v_{j_i}\)
is the \(j_i\) entry of \(\vec{v}\), and let \(\eta\) be the element of \(\RL\) whose
coordinate vector with respect to \(B\) is
\(\vec{w}\).  We define the reduction of \(\beta(x)\) by \(\Aadd\) to be the unique 
\(\gamma(x)\in\RK[x]\) such that \(\deg(\gamma)<f\) and \(\gamma(\overline\xi)=\eta\).
\end{definition}


\subsection*{Reduction of {\O}ystein polynomials}\label{sec red form}

We now generalize the reduction algorithm for Eisenstein polynomials
to {\O}ystein polynomials.  Let \(\varphi\in\KK[x]\) be a \(\nu\)-{\O}ystein polynomial, let
\(\alpha\) be a root of \(\varphi\), and set $\KL=\KK(\alpha)$.
We investigate the effect of the transformation $\alpha\mapsto\alpha+\gamma\nu(\alpha)^m$ on the coefficients of the respective minimal polynomials.  Generalizing Proposition \ref{prop red} to {\O}ystein polynomials we obtain:

\begin{theorem}\label{theo red}
Let \(\varphi\in\OK[x]\) be a \(\nu\)-{\O}ystein polynomial written as
\[
\varphi(x)= \nu(x)^{e}+\sum_{i=0}^{e-1}\varphi^*_i(x) \nu(x)^i 
\text{ with }
\varphi^*_i(x)=\sum_{j=0}^{f-1} x^j \sum_{k=1}^\infty \varphi^*_{i,j,k} \piK^k
\text{ and }\varphi_{i,j,k}^*\in \RepK.
\]
Let \(\alpha\) be a root of $\varphi$ and set \(\KL=\KK(\alpha)\).  Let $\beta = \alpha+\gamma(\alpha)\nu(\alpha)^{m+1}$
for some \(m\in\N_0\) and \(\gamma\in\OK[x]\) with \(\deg(\gamma)<f\) and 
\(v_\KL(\gamma(\alpha))=0\) such that $v_L(\nu(\beta))=1$.
Then the minimal polynomial \(\psi\in\OK[x]\) of $\beta$ over \(\KK\) is of the form 
\[
\psi(x)= \nu(x)^{e}+\sum_{i=0}^{e-1}\psi^*_i(x) \nu(x)^i
\text{ with }\psi^*_i(x)=\sum_{j=0}^{f-1} x^j \sum_{k=1}^\infty \psi^*_{i,j,k} \piK^k
\text{ and }
\psi^*_{i,j,k}\in \RepK.
\]
Let \(\rho(x)=\frac{\varphi(\nu(\alpha)x+\alpha)}{\nu(\alpha)^e}\) be the $\nu$-ramification polynomial of \(\varphi\),
and for \(m\in\mathbb{N}_0\) set 
\(
A^+_{m}(x)={\frac{\rho(\nu(\alpha)^m x)}{\nu(\alpha)^{e\Phi_{\KL/\KK}(m)}}}.
\)
Then
\begin{enumerate}
\item $\varphi_{i,j,k}^*=\psi_{i,j,k}^*$ for all $i,k$ such that 
\(i+ek<e(\Phi_{L/K}(m)+1)\).
\item Set \(h=e\Phi_{\KL/\KK}(m)\), \(\widetilde{\imath} = h \bmod e\) and \(\widetilde{k} = (h\bdiv e)+1\).
Let $\eta \in \OK[x]$ satisfy $\eta(\alpha) = \piK/\nu(\alpha)^e$.  Then
\[
\sum_{j=0}^{f-1} \left(\overline{\varphi^*_{\widetilde{\imath},j,\widetilde{k}}}-\overline{\psi^*_{\widetilde{\imath},j,\widetilde{k}}}\right)\overline{\alpha}^j
=\overline{\eta}(\overline{\alpha})^{-k} \Aadd_{m}(\overline{\gamma}(\overline{\alpha})).
\]
\end{enumerate}
\end{theorem}
Note that the condition $v_L(\nu(\beta))=1$ is always satisfied for \(m>0\),
since in this case 
\(\nu(\beta)=\nu(\alpha+\gamma(\alpha)\nu(\alpha)^{m+1})=\nu(\alpha)+\nu(\alpha)^2 \delta\) for some \(\delta\in\KL\) with \(v_\KL(\delta)\ge 0\).
\begin{proof} We have
\(
\varphi(\beta)-\psi(\beta) = \varphi(\beta) = \varphi(\alpha+\gamma(\alpha)\nu(\alpha)^{m+1}) =
\rho(\gamma(\alpha)\nu(\alpha)^m)\nu(\alpha)^e.
\)
and also
\[
\varphi(\beta)-\psi(\beta)
= \sum_{i=0}^{e-1}\left(\varphi^*_i(\beta)-\psi^*_i(\beta)\right) \nu(\beta)^i
= \sum_{i=0}^{e-1} \sum_{j=0}^{f-1} \beta^j
\sum_{k=1}^\infty \left(\varphi^*_{i,j,k}-\psi^*_{i,j,k}\right) \nu(\beta)^i\piK^k.
\]
Combine the two expressions for \(\varphi(\beta)-\psi(\beta)\) and divide by \(\nu(\alpha)^{e+e\Phi_{\KL/\KK}(m)}\) to get
\begin{equation} \label{combine}
\sum_{i=0}^{e-1} \sum_{j=0}^{f-1}\beta^j
\sum_{k=1}^\infty \left(\varphi^*_{i,j,k}-\psi^*_{i,j,k}\right) \frac{\nu(\beta)^i\piK^k}{\nu(\alpha)^{e+e\Phi_{\KL/\KK}(m)}}
= \frac{\rho(\gamma(\alpha)\nu(\alpha)^m)}{\nu(\alpha)^{e\Phi_{\KL/\KK}(m)}}. 
\end{equation}
Note that we have \(\beta\sim\alpha\) and that since \(v_\KL(\pi_\KK)=e\) there exists
\(\eta\in\OK[x]\) with \(\eta(\alpha)=\pi_\KK/\nu(\alpha)^e\).  

Suppose \(m=0\).  Since \(\Phi_{L/K}(0)=0\),
statement (1) holds vacuously.  Furthermore, we have 
\(\widetilde{k}=1\), \(\widetilde{\imath}=0\), so (\ref{combine}) gives
\[
\sum_{j=0}^{f-1} \alpha^j \left(\varphi^*_{0,j,1}-\psi^*_{0,j,1}\right) \eta(\alpha)
\equiv \rho(\gamma(\alpha)) \equiv A^+_{0}(\gamma(\alpha))\pmod{\nu(\alpha)},
\]
which implies statement (2) for this case.

For \(m>0\) we have \(\nu(\beta)=\nu(\alpha+\gamma(\alpha)\nu(\alpha)^{m+1})\sim\nu(\alpha)\). 
Therefore by (\ref{combine}) 
\[
\sum_{i=0}^{e-1} \sum_{j=0}^{f-1} \alpha^j
\sum_{k=1}^\infty \left(\varphi^*_{i,j,k}-\psi^*_{i,j,k}\right) 
\eta(\alpha)^k \frac{\nu(\alpha)^{i+ek}}
{\nu(\alpha)^{e+e\Phi_{L/k}(m)}}
\equiv A^+_{m}(\gamma(\alpha)) \pmod{\nu(\alpha)}.
\]
If \(i+ek<e(\Phi_{L/K}(m)+1)\) then \(\varphi^*_{i,j,k}=\psi^*_{i,j,k}\) since
\(A^+_{m}(\gamma(\alpha))\in\OL\).  This gives statement (1).
All the terms on the left
with \(i+ek>e(\Phi_{L/K}(m)+1)\) are congruent to \(0\pmod{\nu(\alpha)}\).
By considering the remaining terms (those with
$i=\widetilde{\imath}$ and $k=\widetilde{k}$) we get statement (2).
\end{proof}

\antsonly{
Let \(\varphi\in\OK[x]\) be a \(\nu\)-{\O}ystein polynomial with root \(\alpha\) and set \(\KL=\KK(\alpha)\). 
Then every $\nu$-{\O}ystein polynomial that defines \(\KL/\KK\) can be obtained from $\varphi$ 
by a sequence of transformations of the type described in Theorem \ref{theo red}.
}
\arxivonly{
\begin{lemma}\label{lem all eisenore}
Let \(\varphi\in\OK[x]\) be a \(\nu\)-{\O}ystein polynomial with root \(\alpha\) and set \(\KL=\KK(\alpha)\).  
Then every $\nu$-{\O}ystein polynomial that defines \(\KL/\KK\) can be obtained from $\varphi$ 
by a sequence of transformations of the type described in Theorem \ref{theo red}.
\end{lemma} 

\begin{proof}
Let \(\psi(x)\in\OK[x]\) be a \(\nu\)-{\O}ystein polynomial such that 
\(\KK[x]/(\psi)\cong\KL\) and let \(\beta\in\KL\) be a root of \(\psi\) such that 
\(\KL=\KK(\beta)\).  Since \(\varphi(x)\equiv\psi(x)\equiv\nu(x)^e\pmod\piL\), \(\alpha\) is
congruent modulo \(\piL\) to a conjugate \(\beta'\) of \(\beta\).  Set \(\alpha_0=\alpha\)
and define \(\alpha_1,\alpha_2,\dots\) recursively by choosing \(\gamma_m\in\OK[x]\) such that
\(\alpha_m=\alpha_{m-1}+\gamma_m(\alpha_{m-1})\nu(\alpha_{m-1})^m\) satisfies 
\(\alpha_m\equiv\alpha'\bmod{\piL^{m+1}}\).   Note that since
\(\nu(\alpha_m)\equiv\nu(\alpha')\bmod{\piL^{m+1}}\) we have \(v_\KL(\nu(\alpha_m))=1\). 
Let \(\varphi_m\) be the minimal polynomial of \(\alpha_m\) over \(\KK\).  Then \((\varphi_m)\)
converges coefficientwise to the minimal polynomial of \(\beta'\), which is \(\psi\).  
\end{proof}}
    
Most of the coefficients of an {\O}ystein polynomial can be easily reduced.
We write \(\left(-p^{w_1},v_\KL(\rho_{p^{w_1}})\right)\) and \(\left(-1,v_\KL(\rho_1)\right)\)
for the endpoints of the segment of \(\rampol(\varphi)\) with steepest slope
 \(\widetilde{m}=\frac{v_\KL(\rho_1)-v_\KL(\rho_{p^{w_1}})}{p^{w_1}-1}\).  Let \(m\ge\widetilde{m}\)
 be an integer
and suppose the minimum in (\ref{eq hasse-herbrand 1}) is not attained for $i=1$.  Then there is
\(2\le j\le n\) such that \(v_\KL(\rho_j)+{m}\cdot j< v_\KL(\rho_1)+{m}\).  Thus
\(\frac{v_\KL(\rho_1)-v_\KL(\rho_j)}{j-1}>m\ge \widetilde{m}\), which contradicts the assumption
that \(\widetilde{m}\) is the steepest slope in \(\rampol(\varphi)\).
Hence for \(m\ge\widetilde{m}\) we have \(e\Phi_{\rampol(\varphi)}(m)=v_\KL(\rho_1)+m\).
We get the following version of the Krasner bound (see, for instance, \cite[Lemma~2.3]{pauli-sinclair}):

\begin{corollary}\label{coro easy}
Let \(\psi(x) = \nu(x)^{e}+\sum_{i=0}^{e-1}\psi^*_i(x) \nu(x)^i\in\OK[x]\) and
\(\varphi(x)= \nu(x)^{e}+\sum_{i=0}^{e-1}\varphi^*_i(x) \nu(x)^i\)
be \(\nu\)-{\O}ystein polynomials in \(\OK[x]\).  Assume that 
\(\varphi^*_{i,k}(x)=0\) for all \(i,k\) such that \(i+ek>e(\Phi_{L/K}(\widetilde{m})+1)\) and
\(\varphi^*_{i,k}(x)=\psi^*_{i,k}(x)\) for all other \(i,k\).  Then \(\KK[x]/(\psi)\cong\KK[x]/(\varphi)\).
\end{corollary}
Since the transformation used in Algorithm \ref{alg polred0ram} (3)
is not directly applicable for {\O}ystein polynomials, we
apply Theorem \ref{theo red} for the horizontal segment. 


\Algo{polred0}{alg polred0}
{\(\nu\)-{\O}ystein $\varphi(x)\in\OK[x]$ with $\nu$-ramification polynomial \(\rho\in\OL[x]\)}
{
\(\{\psi\in\OK[x]: \psi=\chr_{\KL/\KK}(\alpha+\theta\nu(\alpha))
\text{ with }\theta\in\RepLur,\;
\log_{\overline{\xi}}\overline\psi_{0,1} \text{ minimal}\}\)}
{
Denote by \(\overline\xiL\in\RL\) a root of \(\overline\nu\).
\begin{enumerate}
\item \(\eta\gets\frac{\piK}{\nu(\alpha)^e}\)
\item
\(
\Aadd_0(x)\gets\overline{
\left(
\frac{\varphi(\nu(\alpha)z+\alpha)}{\cont{\nu(\alpha)}{\varphi(\nu(\alpha)x+\alpha)}}
\right)}
\)
\item 
\(r\gets \min\left\{\log_{\xiRL} \left(\sigma\left(\overline\varphi_{0,1}(\xiRL)\right)\cdot\xiRL^{e\cdot l} \right): 0\le l<\card\lambda-1,\,\sigma\in\Aut(\RL/\RK)\right\} \) 
\item \(\overline{B}(x)\gets\overline\eta^{-1}\Aadd_0(x)-\overline{(\varphi_{0,1}-\xiL^r)}\)
\item \(\Theta\gets\{\theta\in\RepLur:\overline{B}(\overline\theta)=0\}\)
\item return \(\{\chr_{\KL/\KK}(\alpha+\theta\nu(\alpha)):\theta\in\Theta\}\)
\end{enumerate}
}

%



For \(m\in\N\) not covered by Corollary \ref{coro easy} we apply the following algorithm.  It can be used 
for all {\O}ystein polynomials, including Eisenstein polynomials, which are obtained 
by taking \(\nu(x)=x\).
\Algo{polredslope}{alg polredslope}
{\(\varphi(x)=\nu(x)^{e}+\sum_{i=0}^{e-1}\varphi^*_i(x)\nu(x)^i\) with root $\alpha$ and \(m\in\N\)}
{The set \(M\) of reduced polynomials obtained by substitutions of the form \(\alpha \mapsto \alpha + \theta \cdot  \nu(\alpha)^{m+1}\), for \(\theta\in\RepLur\)}
{
\vspace{0.1in}
\begin{enumerate}
\item \(\eta\gets\piK/\nu(\alpha)^e\)
\item \(\Aadd_m \gets\overline{\frac{\varphi(\alpha+\nu(\alpha)^{m+1} x)}{\nu(\alpha)^{e\Phi_{\KL/\KK}(m)}}}\)
\item \(i \gets e\Phi_{\KL/\KK}(m) \bmod e\), \(k\gets (e\Phi_{\KL/\KK}(m) \bdiv e)+1\)
%
\item let \(\overline\delta\) be the reduction of \(\overline{\varphi^*_{i,k}(\alpha)}\) by \(\Aadd_m\), see Definition \ref{def red cokernel 2}
\item \(\Theta\gets \left\{\theta\in\RepLur : \overline{\varphi^*_{i,k}(\alpha)}-\overline\delta=
\overline{\eta}^{-k}\cdot\Aadd_m(\overline\theta)\right\}\)
\item \(M\gets \left\{\chr_{\KL/\Qp}(\alpha+\theta\cdot \nu(\alpha)^{m+1}) : \theta \in \Theta\right\}\)
\end{enumerate}
}

\subsection*{Tamely ramified extensions}
Let 
\(\varphi(x)=\nu(x)^\epsilon+\sum_{i=0}^{\epsilon-1}
\varphi^*_i(x)\nu(x)^i\in\OK[x]
\)
be \(\nu\)-{\O}ystein with \(p\nmid\epsilon\)
and \(\KL_0=\KK[x]/(\varphi)\).
Because \(\epsilon\Phi_{\KL_0/\KK}(m)=m\) and 
the steepest slope is \(\widetilde{m}=0\),
by Corollary \ref{coro easy} for
\(\widetilde\varphi(x)=\nu(x)^\epsilon+\varphi_{0,1}(x)\piK\)
we have \(\KL_0\cong \KK[x]/(\widetilde\varphi)\).
By Lemma \ref{lem const coeff}, for 
\(\psi(x)=x^{\epsilon}+\psi_{0,1}\piK\in\Lur[x]\) 
with \(\psi_{0}^*\sim\varphi_{0}^*(\xiL)\) and \(\psi(\beta)=0\) we have 
\(\Lur(\beta)=\KK(\xiL,\beta)\cong \KL_0\).
Let 
\(
\sigma:\Lur\to\Lur \text{ with }\sigma:\xiL\to\xiL^{q^u} \text{ for some } 0\le u<f 
\)
and let \(\beta^{(\sigma)}\) be a root of 
\(\psi^{(\sigma)}(x)=x^{\epsilon}+\psi^{(\sigma)}_{0,1}\piK\in\Lur[x]\).  
Then
\(\KK(\xiL,\beta^{(\sigma)})\cong\KK(\xiL,\beta)\).
Substitution of \(x\) by \(\xiL^{-l} x\) in \(\psi^{(\sigma)}\) and multiplication by \(\xiL^{l\epsilon}\) yields 
\(x^{\epsilon}+\psi^{(\sigma)}_{0,1}\xiL^{l\epsilon}\piK\).
To satisfy Condition \ref{con dist const oystein}, we find \(l\) and \(\sigma\) such that 
\[
\log_{\xiRL}\left(\overline\psi^{(\sigma)}_{0,1}\xiL^{l\epsilon}\right)=(a q^u+l\epsilon)\bmod (q^f-1) \] 
is minimal, where \(a=\log_{\xiRL}\overline{\psi}_{0,1}\).
As in Lemma \ref{lem min} set \(g=\gcd(\epsilon,q^f-1)\). 
We now find \(0\le u< f\) such that \(r=aq^u\bmod g\) is minimal.
Then the distinguished defining Eisenstein 
polynomial of \(L_0/\Lur\) is \(\psi(x)=x^{\epsilon}+\xiL^r\piK\) and
the distinguished defining {\O}ystein polynomial of \(\KL_0/\KK\) is 
\[
\varphi(x)=\nu(x)^\epsilon+\varphi^*_{0,1}(x)\piK \text{ with }\deg\varphi^*_{0,1}<e\text{ and }\overline\varphi^*_{0,1}(\xiRL)=\xiRL^r.
\]
\arxivonly{

\begin{example}\label{ex f=2 e=8}
We find distinguished defining polynomials for all extensions \(\KL/\Q_3\) of degree 16 with
inertia degree \(f=2\) and ramification index \(\epsilon=8\).
We define the unramified part of the extension as \(\Lur=\Q_3(\xi)\) where \(\xi\)
is a root of the Conway polynomial \(\nu_2(x)=x^2+2x+2\).
Over \(\Lur\) each of our extensions is defined by a polynomial of the form
\[
\psi_c(x)=x^8+\xi^c \cdot 3\text{ where }0\le c \le 7.
\]
Since \(g=\gcd(8,9-1)=8\) these define distinct extensions over \(\Lur\).  
Because
\(\Aut(\RL)=\langle \overline\xi\mapsto \overline\xi^3\rangle \),
the conjugates of 
\(\psi_c\) over \(\Lur\) are of the form \(\psi_{d}(x)=x^8+\psi_{d,0}\) where 
\(\psi_{d,0}\sim \xi^d\cdot 3\) with \(d\equiv 3^kc\pmod{8}\) for \(0\le k< 2\).
Because
\(\min\{0\cdot 3^k:0\le k< 2\}=0\), 
\(\min\{3^k\bmod 8:0\le k< 2\}=\min\{1,3\}=1\), 
\(\min\{2\cdot 3^k\bmod 8:0\le k< 2\}=\min\{2,6\}=2\), 
\(\min\{4\cdot 3^k\bmod 8:0\le k< 2\}=4\), and
\(\min\{5\cdot 3^k\bmod 8:0\le k< 2\}=\min\{5,7\}=5\)
the distinguished defining polynomials of our extensions over \(\Q_3\) are
\[
\nu(x)^8+3\varphi_{0,1}\in\Z_3[x]\text{ where }
\varphi_{0,1}\in \{1,x,x+1,2,2x\}
\]
such that \(\log_{\overline\xi}\overline\varphi_{0,1}(\overline\xi)\in\{0,1,2,4,5\}\).
\end{example}
}

\subsection*{Final Distinction}\label{sec final dist}
Since the reduction algorithm yields a set of reduced polynomials rather than a unique reduced polynomial, we need to distinguish one of these.  We choose the polynomial that is minimal with
respect to the following ordering.

\begin{definition}\label{def final dist}
Let \(\nu\in\OK[x]\) be monic and irreducible modulo \(\piK\) 
and write
\[
\varphi(x)=\nu(x)^e+\sum_{i=0}^{e-1}
\left(
\sum_{k=1}^\infty \varphi^*_{ik}(x)\piK^k
\nu(x)^i
\right)
\in\OK[x]
\text{ where }
\varphi^*_{ik}\in\RepK[x],\,\deg\varphi^*_{ik}<f
\]
Given
\(\overline\gamma,\overline\delta\in\RL\) we write
\(\overline\gamma>_{\xiRL} \overline\delta\) when \(\log_{\xiRL}\overline\gamma>\log_{\xiRL}\overline\delta\) (compare 
Definition \ref{def dist res}).
For two \(\nu\)-{\O}ystein polynomials \(\varphi,\psi\in\OK[x]\) 
we write \(\varphi >_{\xiRL} \psi\) when
\(\overline\varphi^*_{ik}(\xiRL) = \overline\psi^*_{ik}(\xiRL)\) for all
\(k<k_0\) and all \(0\le i < e\), \(\overline\varphi^*_{ik_0} = \overline\psi^*_{ik_0}\) for 
\(i<i_0\), and \(\overline{\varphi}^*_{i_0k_0}(\xiRL) >_{\xiRL} \overline{\psi}^*_{i_0k_0}(\xiRL)\).
\end{definition}

\subsection*{Main Algorithm}\label{sec algo}

\begin{definition}\label{con main}
A \(\nu\)-{\O}ystein polynomial \(\varphi\in\OK[x]\)  is \emph{distinguished} if it satisfies the following, given in order of precedence.
\begin{enumerate}
\item\label{first} \(\varphi\) satisfies Condition \ref{con dist const} or \ref{con dist const oystein}.
\item The tuple of residual polynomials of \(\varphi\) is minimal (Definitions \ref{def dist res} and \ref{def dist res conj},
considering the constraint imposed by (\ref{first})).
\item The coefficients of \(\varphi\) are 
reduced with Algorithm \ref{alg polred0} for slope \(m=0\) and Algorithm
\ref{alg polredslope} for integral slopes \(m > 0 \). 
\item Among the polynomials that satisfy all of the above, \(\varphi\) is smallest according to Definition \ref{def final dist}.
\end{enumerate}
\end{definition}

We summarize all of this in the algorithm \texttt{polredpadic}.

\Algo{polredpadic}{alg polredpadic}
{\(\nu\)-{\O}ystein polynomial \(\varphi(x)\in\OK[x]\) of degree \(f\epsilon p^w\) with \(v_\KK(\disc(\varphi))=f(n+J-1)\)}
{the distinguished \(\nu\)-{\O}ystein polynomial \(\psi\) with \(\OK[x]/(\varphi)\cong \OK[x]/(\psi)\)}
{
\begin{enumerate}
\item compute the $\nu$-ramification polynomial \(\rho\) and ramification polygon \(\rampol\) of \(\varphi\).
where \(\widetilde{m}\) is the steepest slope of \(\rampol\) 

\item \(M_0 \gets \left\{\begin{array}{ll}
\mathtt{polred0ram}(\varphi,\rho) & \text{ if } \varphi \text{ is Eisenstein}\\ 
\mathtt{polred0}(\varphi,\rho) &\text{ otherwise}
\end{array}\right.\)
\item for \(m\in\{1,\dots,\lceil\widetilde{m}\rceil\}\):
\begin{itemize}
\item[\(\bullet\)] \(M_m\gets\emptyset\)
\item[\(\bullet\)] for \(\psi\in M_{m-1}\):
\(M_m \gets M_m \cup\) \(\mathtt{polredslope}(\psi,m)\)
\end{itemize}
\item for \(\psi\in M_{\lceil\widetilde{m}\rceil}\) and \(m>\widetilde{m}\):
\begin{itemize}
\item[\(\bullet\)]  \(\psi^*_{i,k}\gets 0\) where \(i=e\Phi_{\KL/\KK}(m) \bmod e\), \(k=(e\Phi_{\KL/\KK}(m) \bdiv e)+1\)
\end{itemize}
\item return the distinguished polynomial in \(M_{\lceil\widetilde{m}\rceil}\) according to Definition \ref{def final dist}
\end{enumerate}
}

\section{Complexity}\label{sec complexity}

We denote by \(\costmult{R}\) the cost in bit operations of multiplication in a ring \(R\) and by \(\costmult{R,n}\) the cost of multiplying two polynomials of degree up to \(n\) over \(R\).  
When we compute in \(\OK\) to an absolute \(\piK\)-adic precision of \(c\) then
with fast Fourier transform \cite{schoenhage-strassen} we get
\[
\costmult{\OK/(\piK^c)}=\costmult{\RK,c}=\oh{c\cdot\log c\cdot\log\log c\cdot\costmult{\RK}}=\softoh{c \costmult{\RK}}.
\]
Also, if \([\KL:\KK]=f\epsilon p^w\), we have
\(\costmult{\RL}=\costmult{\RK,f}=\softoh{f\costmult{\RK}}\)
and \(\costmult{\OL/\piL^d}=\costmult{\RL,d}=\costmult{\RK,fd}\).
With the Cantor-Zassenhaus algorithm \(\overline{B}\in\RK[z]\) of degree \(n\) can be factored in 
expected \(\costroot(\RK,n)=\softoh{(n^2+n\log(q))\costmult{\RK}}\) bit operations where \(q=\card{\RK}\). 
The discrete logarithm \(\log_{\xiRL}\overline{\gamma}\) can be computed in 
\(\costlog(\RL):=\oh{\costmult{\RL}{q^{f/2}}}\) bit operations with Shanks Baby-Step Giant-Step algorithm,
where \(q^f=\card{\RL}\).

\subsection*{{\O}ystein}
For the computation of characteristic polynomials we use the following.
\begin{proposition}
[compare {\cite[Proposition 4.3.3]{cohen-1}}]\label{prop char poly}
Let \(\alpha_1,\dots,\alpha_n\) be the roots of
\(\varphi(x)=x^n+\sum_{i=0}^{n-1} \varphi_i x^i\in\OK[x]\).
Write \(S(\varphi)=(S_1,\dots,S_n)\) with \(S_k=\sum_{i=1}^n \alpha_i^k\) where we set \(\varphi_i=0\) for \(i<0\).
Then
\begin{enumerate}
\item \(S_k=-k \varphi_{n-k}-\sum_{i=1}^{k-1}\varphi_{n-i} S_{k-i}\) for \(1\le k\le n\) and
\item 
\(\varphi_{n-k}=-S_k/k-\sum_{i=1}^{k-1}(\varphi_{n-i}S_{k-i})/k\) for \(1\le k\le n\)
\setcounter{savedcounter}{\value{enumi}}
\end{enumerate}
Given \((S_1,\dots,S_n)\) we write \(\varphi=X(S_1,\dots, S_n)\) such that \(X(S(\varphi))=\varphi\).\\[1ex]
Let \(\beta\in\OK[x]\) and write \(b^{(k)}_i\) for the \(i\)-th coefficient of \(\beta^k \bmod \varphi\).
\begin{enumerate}
\setcounter{enumi}{\value{savedcounter}}
\item \(\chr_{\varphi}({{\beta}})=X(T_1,\dots,T_n)\)
where \(T_k=nb^{(k)}_0+\sum_{i=1}^{n-1}S_i b^{(k)}_i\)

\item 
\(\chr_{\varphi}\left(\frac{\beta}{\piK^u}\right)(x)=
\frac{\psi(x)}{\cont{\piK}{\psi}}\)
where \(\psi(x)=\chr_\varphi(\beta)(\piK^u x)\).
\end{enumerate}
\end{proposition}

\begin{lemma}\label{prop comp char poly}
Let \(\varphi\in\OK[x]\) be monic and irreducible of degree \(n\).  Let
\(\gamma=\beta \piK^{-u}\in\KK[x]\) where \(\beta\in\OK[x]\).
Following Proposition \(\ref{prop char poly}\), \(\chr_{\varphi}({\gamma})\) can be computed to a \(\piK\)-adic precision 
of \(b\) in \(\costchar(\varphi,u,b):=\oh{\costmult{n}\costmult{\OK/(\piK^c)}}\) bit operations, where 
\(c=b+\max\{n\cdot u,e_\KK \lceil\log_p n\rceil\}\).
\end{lemma}

\begin{proof}
We go through the parts of Proposition \ref{prop char poly} in the order in which they would be used in the computation of
a characteristic polynomial.
\begin{enumerate}
\item Let \(E(t)=t^{n}\varphi(1/t)\) and \(S(t)=\sum_{k=0}^\infty S_kt^k\).
        Then \(S(-t)\equiv E^{\prime }(t){E(t)^{-1}}\pmod{t^{n+1}}\).
        Using Newton inversion we compute the inverse \(E(t)^{-1}\) of \(E(t)\) in \(\oh{n\log n}\)
        and \(E'(t) \cdot E(t)^{-1}\bmod{t^{n+1}}\) in \(\costmult{n}\) operations in
        \(\OK\).
\item[(3)] With fast multipoint evaluation, all \(T_k\) can be computed \(\oh{\costmult{n}\log n}\)
operations in \(\OK\).
\item[(2)] Similarly as in (1) this takes \(\costmult{n}\) operations in \(\OK\).
\item[(4)] When the coefficients of \(\gamma\) are not integral, to allow for the factor \((\piK^u)^n\) 
we increase the precision by \(n\cdot u\) to avoid precision loss.\qedhere
\end{enumerate}
\end{proof}

Our implementation of Proposition \(\ref{prop char poly}\) uses \(\oh{n^2}\) operations in \(\OK\), which is all that is needed for degrees of the polynomials that we are considering.

\begin{proposition}
Given an irreducible polynomial \(\psi\in\OK[x]\) 
of degree $n$ with discriminant exponent \(b=v_\KK(\disc(\psi))\),
Algorithm \(\ref{alg eisenore}\) returns a \(\nu\)-{\O}ystein polynomial
\(\varphi\in\OK[x]\) such that \(\KK[x]/(\psi)\cong\KK[x]/(\varphi)\) in \(\softoh{nb^2\costmult{\RK}}\) bit operations. 
\end{proposition}

\begin{proof}
Let \(c=\lfloor b/n\rfloor+1\). 
The costs of the steps of the algorithm are:
\begin{enumerate}
\item[(1,3)] The Okutsu invariants of \(\psi\) can be computed in 
\(\softoh{(n^2+n b^2)\costmult{\RK}}
\) 
bit operations \cite{bauch-nart-stainsby}.  
This includes the computation of \(\Pi\) and \(\Gamma\).
\item[(2)] By our assumption we can obtain \(\nu\) by lookup.
\item[(4)] The precision \(c\) required in (2) is also an upper bound for the minimum of 
the valuations of the coefficients of \(\Gamma\) and \(\Pi\).  To avoid 
precision loss we increase the precision by \(cf\), so \(\alpha\) is computed in 
\(\oh{f\costmult{\OK/(\piK^{c+cf}),n}}
=\softoh{f(c+cf)n\costmult{\RK}}
\) bit operations.
\item[(5)] With the bound for the valuation of the denominator from (4) this takes
\(\costchar(\psi,c,c)
=\oh{\costmult{n}\costmult{\OK/(\piK^{c+nc})}}
=\softoh{n(c+nc)\costmult{\RK}}
\)
bit operations.
\end{enumerate}
As \(\psi\) is irreducible we have 
\(b=v_\KK(\disc(\psi))\ge v_\KK\left(\disc(\KK[x]/(\psi))\right)=f(e+j-1)\) for some 
\(j\in\N_0\) where \(e\) and \(f\) are the ramification index and inertia degree of 
\(\KK[x]/(\psi)\).
Thus in total we get
\(\softoh{({n b^2})\costmult{\RK}}\) bit operations.
\end{proof}

\subsection*{Reduction}
Let \(\varphi\in\OK[x]\) be \(\nu\)-{\O}ystein of degree \(n=fe\) with inertia degree \(f\) and ramification index \(e=\epsilon p^w\) with \(\gcd(p,\epsilon)=1\) and \(v_\KK(\disc(\varphi))=f(n+J-1)\).
Let  \(\KL=\KK[x]/(\varphi)\).
Let \(\rho\) be the $\nu$-ramification polynomial of \(\varphi\), let 
\((-p^{w_u},v_\KL(\rho_{p^{w_u}}))\) and
\((-1,v_\KL(\rho_{1}))=(-1,J)\) be the endpoints of the steepest segment of the ramification polygon \(\rampol\),
and let \(\widehat{m}\) be the slope of this segment.
It follows from Corollary \ref{coro easy} that it is sufficient to determine the reduced polynomials  over 
\(\OK/\left(\piK^{\widehat{c}}\right)\) for any \(\widehat{c}\in\N\) such that
\begin{equation}\label{eq prec}
\widehat{c}>\Phi_{\rampol}(\widehat{m})+1=v_\KL(\rho_1)+\widehat{m}+1=J+\widehat{m}+1.
\end{equation}
We use the \(\piK\)-adic precision of \(\widehat{c}\) throughout this section.

\begin{lemma}\label{lem cost 0ram}
Let \(\varphi\in\OK[x]\) be an Eisenstein polynomial of degree \(e=\epsilon p^w\).
Algorithm \(\ref{alg polred0ram}\) runs in \(\softoh{eq\widehat{c}\, \costmult{\RK}}\) bit operations.
\begin{proof}
The costs of the steps of the algorithm are:

\begin{enumerate}
\item The coefficients of \(A_{m_1},\dots, A_{m_u}\) can be read directly from 
\(\rho\).
\item We need to compute up to \(w\) discrete logarithms 
at \(\costlog(\RK)=\oh{\costmult{\RK}\sqrt{q}}\) bit operations each and apply Lemma \ref{lem min} up to \(w\) times 
at \(\oh{\log(e)\log( q^f )}\) bit operations each.
\item Because \(\card{D}\le q \), this can be completed in 
\(\oh{eq\widehat{c}\,\costmult{\RK}}\) bit operations.
\item There are at most \(q\) different \(\psi\), each computed in \(e\costmult{\RK,\widehat{c}}\) bit operations.
\end{enumerate}
In total, this algorithm can be completed in 
\(\softoh{eq\widehat{c}\,\costmult{\RK}}\) 
bit operations.
\end{proof}
\end{lemma}

\begin{lemma}\label{lem cost 0}
Let \(\varphi\in\OK[x]\) be \(\nu\)-{\O}ystein of degree \(n=f\epsilon p^w\).
Algorithm \(\ref{alg polred0}\) returns the set of slope \(0\) reduced polynomials in \(\softoh{\bigl({q^{f/2}}+n^2(\widehat{c}+e_\KK)\bigr)\costmult{\RK}}\) 
bit operations.
\end{lemma}
\begin{proof}  
The costs of the steps of the algorithm are:
\begin{enumerate} 
\item Fast exponentiation computes \(\eta\)  in \(\oh{\log(e)\costmult{\RL,n\widehat{c}}}\) 
bit operations.
\item With (\ref{eq rho coeff}) the relevant coefficients of the $\nu$-ramification polynomial can be computed in 
\(\oh{n (p^w-1)\epsilon\cdot\costmult{\RK,\widehat{c}}}\) bit operations.
\item Evaluating \(\overline\varphi^*_{0,1}(\xiRL)\) takes 
\(\oh{f\costmult{\RL}}\) 
bit operations, and its discrete logarithm can be computed in \(\oh{q^{f/2}\costmult{\RL}}\) bit operations.
Since \(\Aut(\RL/\RK)\) is generated by the Frobenius automorphism \(\xiRL\mapsto\xiRL^q\) the discrete logarithms
of the conjugates only need additional \(\oh{f\costmult{\Z/(q^f-1)}}\) bit operations.
Computing the minimum for each of the \(f\) conjugates separately and then taking the minimum  takes 
\(\oh{f\log(e)\log( q )}\)
bit operations with Lemma~\ref{lem min}. Since $f = O(\log(p^f))$, it can be omitted.
\item[(4,5)] We have \(\deg\Aadd_0=\deg\overline{B}=(\epsilon-1)p^w\). 
We find the roots of \(\overline{B}\) in expected 
\(
\costroot(\RL,(\epsilon-1)p^w)=
\softoh{(((\epsilon-1)p^w)^2+(\epsilon-1)p^w f\log q)f\costmult{\RK}}
\)
bit operations.
\item[(6)]
Because \(\overline{B}\) has \(\card\Theta\le\deg\overline{B}=(\epsilon-1)p^w\)  roots,  
this step takes up to 
\((\epsilon-1)p^w\,\costchar(\KK,0,\widehat{c})=
\softoh{n^2(\widehat{c}+e_\KK)\,\costmult{\RK}}\) bit operations.
\end{enumerate}
Adding these yields the statement of the lemma.
\end{proof}

\begin{lemma}\label{lem cost slope}
Let \(\varphi\in\OK[x]\) be \(\nu\)-{\O}ystein of degree \(n=f\epsilon p^w\) and let \(m\in\N\).
Denote by \(-p^{u}\) and \(-p^{v}\) be the abscissas of the endpoints of the segment of slope \(m\) 
on \(\rampol(\varphi)\) (if there is no segment of slope $m$ then $u=v$).
Algorithm \(\ref{alg polredslope}\) runs in 
\(\softoh{(f^3+p^{u-v} \costmult{n}(\widehat{c}+e_\KK))\,\costmult{\RK}}\) 
bit operations.
\begin{proof} 
The costs of the steps of the algorithm are:
\begin{enumerate}
\item Fast exponentiation computes \(\eta\)  in \(\oh{\log(e)\costmult{\RL,n\widehat{c}}}\) 
bit operations.
\item[(2,3)] With (\ref{eq rho coeff}) this takes \(\oh{n\costmult{\RK,\widehat{c}}}\) bit operations
\item[(4)] 
The \(\nu\)-expansion of \(\varphi\) can be computed in \(\oh{fn\cdot\costmult{\RK,\widehat{c}}}\)
bit operations and the coefficient \(\varphi^*_{i,k}\) is computed in at most 
\(\oh{f\widehat{c}\cdot\costmult{\RK}}\)
bit operations.
By Lemma \ref{lem additive} \(\Aadd_m\) has at most \(u-v+1\) non-zero coefficients.
An \({f\times f}\) matrix over \(\RK\) representing \(\Aadd_m:\RL\to\RL\) where \(\Aadd_m\in\RL\) and its row echelon form  can be computed in \(\oh{f(u-v+1)\costmult{\RL,f}+f^3\costmult{\RK}}\) bit operations.
\item[(5)]  
We use linear algebra to find \(\Theta\).  The system of linear equations can be solved in \(f^3\costmult{\RK}\) bit operations.
\item[(6)] We have \(\card{\Theta}\le p^{u-v}\).  Thus, we need to compute at most \(p^{u-v}\) 
characteristic polynomials which can be done with \(\softoh{p^{u-v} \costmult{n} (\widehat{c}+e_\KK)\costmult{\RK}}\) 
bit operations (compare Lemma \ref{lem cost 0} (6)) \arxivonly{where the factor \(\log_p n\)
is swallowed by the \(\softohalone\)}. 
\end{enumerate}
Adding the above we obtain the statement of the Lemma.
\end{proof}
\end{lemma}

\begin{proposition}
Let \(\varphi\in\OK[x]\)  of degree \(n=f\epsilon p^w\) with \(v_\KK(\disc(\varphi))=f(n+J-1)\) be \(\nu\)-{\O}ystein.
Algorithm \(\ref{alg polredpadic}\) returns the distinguished defining polynomial of \(\KK[x]/(\varphi)\) in 
\(\softoh{\bigl(q^{f/2}+eq\widehat{c}+n^2(\widehat{c}+e_\KK)+n\widehat{c}^2\bigr)\costmult{\RK}}\) bit operations
where \(\widehat{c}=J+\widehat{m}+1\) as in \((\ref{eq prec})\).
\end{proposition}

\begin{proof} 
The costs of the steps of the algorithm are:
\begin{enumerate}
\item 
By Lemma \ref{lem additive}, the coefficients that contribute to the ramification polygon are \(\rho_{i}\) for \(i\in\{1,p,\dots,p^w\}\).  They can be computed using (\ref{eq rho coeff}) and their valuations can be computed by repeated squaring of the uniformizer in a total of \(\oh{w \cdot\log_2 J\cdot \costmult{\OK/(\piK)^{J+1}}}\) bit operations.
\item The slope 0 reduction step takes 
\(\softoh{\bigl(q^{f/2}+n^2(\widehat{c}+e_\KK)\bigr)\costmult{\RK}}\)  bit operations by Lemma \ref{lem cost 0} if $f > 1$
or \(\softoh{eq\widehat{c}\,\costmult{\RK}}\) by Lemma \ref{lem cost 0ram} if $f=1$.
\item  
For \(m\in\{1,\dots,\widehat{m}\}\) the slope \(m\)
reduction step only results in additional polynomials 
when \(\rampol\) has a segment of slope \(m\).
Write \((-p^{w_i},v_\KL(\rho_{p^{w_i}}))\) and
\((-p^{w_{i+1}},v_\KL(\rho_{p^{w_{i+1}}}))\)
for the endpoints of the \(i\)-th segment of \(\rampol\).
It follows from (6,7) in the proof of Lemma \ref{lem cost slope} that
the number of reduced polynomials after this step 
is at most \(\prod_{j=1}^{u-1} (p^{w_{j+1}}/p^{w_{j}})=p^w\).
Since the number of segments of \(\rampol(\varphi)\) is at most \(w\)
the number of bit operations for this step is less
than 
\(\softoh{(wf^3+p^w n(\widehat{c}+e_\KK))\,\costmult{\RK}}\)
by Lemma \ref{lem cost slope}.
\item In practice {\(\widehat{c}\)} is chosen as the precision
 of these computations, so that at most \(n-2\) coefficients  to 0.
\item To compare two polynomials \(\varphi,\psi\in\OK[x]\) with respect to the ordering from Definition \ref{def final dist}
we compute their difference and compute the valuations of its coefficients in \(\oh{n\log \widehat{c}\,\costmult{\OK/(\piK^{\widehat{c}})}}\) bit operations.
The lowest valuation determines which coefficient we need to consider.  Its \(\piK\)-adic expansions is computed 
in at most \(\widehat{c}\,\costmult{\OK/(\piK^{\widehat{c}})}\) bit operations and the discrete logarithm of the relevant 
coefficient is found in \(\oh{\sqrt{q}}\costmult{\RK}\) bit operations.
With merge sort, we find the smallest reduced polynomial with \(\oh{\log(p^w)p^w}\) of these comparisons.
Thus this step can be completed in
\(\softoh{p^w((n+\widehat{c})\widehat{c}+\sqrt{q})\costmult{\RK}}\) bit operations.
\end{enumerate}
 In total these are
 \(\softoh{\bigl(q^{f/2}+eq\widehat{c}+n^2(\widehat{c}+e_\KK)+n\widehat{c}^2\bigr)\costmult{\RK}}\)  bit operations.
\end{proof}

For example, when $K=\Qp$, taking the largest possible $\widehat{c}$ yields a maximum runtime of $\softoh{p^{1+f/2} + n^4p}$.

\subsection*{Timings}
In the table below we give minimum, maximum, and average 
running times for finding distinguished defining polynomials using our implementation
of Algorithm \ref{alg eisenore} \texttt{oystein} and Algorithm \ref{alg polredpadic} \texttt{polredpadic} in Magma.
In each case, we considered a random defining polynomial (created using Pari's \texttt{poltschirnhaus} function) for each ramified extension of degree \(n\) over \(\Q_p\).  All times are in seconds.

\begin{center}
\begin{tabular}{r|r|r|S|S|S}
\(p\) & \(n\) & \(\#\KL\) &\text{min time} & \text{max time} & \text{avg time} \\\hline 
  2   &   6   &     46     & 0.03  & 0.31  & 0.11  \\
  3   &   6   &     74    & 0.03  & 0.53  & 0.15  \\
  2 & 8 & 1822 & 0.01 & 1.0 & 0.37 \\
  2 & 16 & 890110 & 0.04 & 25 & 12 \\
  2 & 18 & 2990 & 0.09 & 19 & 4.4 \\
  3 & 18 & 130646 & 0.0 5& 98 & 19 \\
  2 & 20 & 314542 & 0.03  & 40  & 11  \\
  5   &   20  &     7586   & 0.11  & 360   & 69  \\
  7   &   21  &   3783 & 0.14  & 190  & 14 
\end{tabular}
\end{center}
The computations were run single-threaded in Magma on a Linux computer with 
AMD 2.0 GHz processors and 2 terabytes of RAM.  Note that, in the reduction 
steps, our current implementation does not use the method analyzed in Lemma 
\ref{prop comp char poly} but uses Magma's \texttt{CharacteristicPolynomial},
which becomes the most expensive part of the computation.

For the task of looking up a field defined by an arbitrary polynomial against a list of $N$ distinguished polynomials, one can either use \texttt{polredpadic} or an implementation of Panayi's root finding algorithm \cite{panayi, pauli-roblot}.  The latter is generally faster when $N$ is small.  But when $N$ is large (say, more than $100{,}000$), \texttt{polredpadic} will be much faster since Panayi's algorithm is a pairwise isomorphism test.

\arxivonly{
\section{Examples} \label{sec examples}

\begin{example}\label{ex false friends}
Let \(\xi\) be a root of \(\nu=x^2 + 2x + 2\in\OK[x]\).
In the table below we list
distinguished relative defining polynomials
\(\psi_i\in\Z_3(\xi)[y]\) and absolute defining polynomials \(\varphi_i\in\Z_3[x]\)
with \(\Q_3(\xi)[y]/(\psi_i)\cong \Q_3[x]/(\varphi_i)\)  along with their ramification polygons and residual polynomials \(\overline{A}\).  The polynomials \(\varphi_1\) and \(\psi_2\) and \(\varphi_2\) and \(\psi_3\)
are `false friends'.  The distinguished polynomial \(\psi_1\) is the reduction of the false friend of \(\varphi_3\).
\begin{center}
\small
\begin{tabular}{l|r|c}
\multicolumn{1}{c|}{Defining Polynomials} &
\multicolumn{1}{c|}{Ramification Polygon} & \(\overline{A}\) \\\hline
\(\psi_1=x^9 + (6\xi + 6)y^7 + 3\xi y^6 + 3\) &
\((-9,0), (-1,7)\) & \((z+\overline\xi^2)\) \\
\(\varphi_1=\nu^9 + (3x + 3)\nu^7 + 3x\nu^6 + 3\) &
\((-18,9),(-9,0),(-1,7)\) &                 \\\hline
\(\psi_2=y^9 + (3\xi + 3)y^7 + 3\xi y^6 + 3\) &
\((-9,0), (-1,7)\) & \((z+\overline\xi^6)\) \\
\(\varphi_2=\nu^9 + (6x + 3)\nu^6 + 3\) &
\((-18,9),(-9,0),(-1,7)\) &                 \\\hline
\(\psi_3=y^9 + (6\xi + 3)y^6 + 3\) &
\((-9,0),(-3,6),(-1,15)\) & \((z^6 + \overline\xi^3\!, \) \\
\(\varphi_3=\nu^9 + (9x + 18)\nu^8 + (24x + 15)\nu^7\) &
\multirow{2}{*}{\((-18,9),(-9,0),(-3,6),(-1,15)\)} & \(\overline\xi^3\!z + \overline\xi^7)\) \\ 
\(\phantom{\varphi_3=} + 3x\nu^6 + (27x + 27)\nu + 3\) & &
\end{tabular}
\end{center}
\end{example}

In the next example, we apply results from the construction of
defining polynomials of all totally ramified extensions of 
\((\pi)\)-adic fields.
The valuations of the coefficients of an Eisenstein polynomial \(\varphi\) 
with a given ramification polygon are given in
\cite[Proposition 3.10]{pauli-sinclair} for example, and \cite[Lemma 4.9]{pauli-sinclair} relates the
coefficients of the residual polynomials of \(\varphi\) to the coefficients of \(\varphi\).



\begin{example}[\href{https://www.lmfdb.org/padicField/family/3.1.18.30b}{LMFDB 3.1.18.30b}]
The extensions of \(\Q_3\) with ramification polygon 
\[
\rampol =\{(-18,0),(-9,0),(-3,3),(-1,13)\}
\]
whose segments have slopes \(0\), \(\frac{1}{2}\) and 5.
Let \(\varphi\) be a defining Eisenstein polynomial of one of these extensions. 
The shape of the ramification polygon yields (compare \cite[Lemma 2.2 and Lemma 3.5]{pauli-sinclair}):
\begin{align}
&v_3(\varphi_{i})\ge 2 \text{ for } i\in\{1,2,4,5,7,8,10,11\} \text{ and }\nonumber\\
&v_3(\varphi_{i})=1 \text{ for }i\in\{0,3,13\} \text{ and } v_3(\varphi_{i})\ge 1  \text{ for }i\in \{6,9,12,14,15,16,17\}
\end{align}
Because \(\overline{2}^{18}=\overline{1}\) and \(\overline{1}^{18}=\overline{1}\) 
the polynomials with \(\varphi_{0,1}=1\) and \(\varphi_{0,1}=2\) define non-isomorphic extensions.

There are two possible residual invariants
\[
\Ainv_\delta=\{(z^9 + 2, 2z^3 + 1,  z^2 + \delta ),(z^9 + 2, 2z^3 + 2, 2z^2 + 2\delta)\} \text{ where } \delta\in\{1,2\}
\]
with distinguished representatives \((z^9 + 2, 2z^3 + 1,  z^2 + \delta )\).
The residual polynomial \(z^2+\delta\) corresponds to the additive residual polynomial 
\(\Aadd_{5,\delta}(z) = z^3 + \delta z\).  With Proposition \ref{prop red},
writing 
$i=e\Phi(m)\bmod 18$ and
$k=(e\Phi(m) \bdiv 18)+1$, we get:
\begin{center}
\begin{tabular}{c||c|c|c||c|c||c|c}
slope $m$ & $e\Phi(m)$ & $i$ & $k$ & $\Aadd_{m,1}$ & \(\varphi_{i,k}\) & $\Aadd_{m,2}$ & \(\varphi_{i,k}\)\\\hline
1   &    24         & 6 & 1   & $z^3$           & 0                 & $z^3$ & 0\\        
2   &    27         & 9 & 1   & $z^3$                & 0                 & $z^3$ & 0\\
3   &    30         & 12 & 1  & $z^3$                & 0                 & $z^3$ & 0\\
4   &    33         & 15 & 1  & $z^3$                & 0                 & $z^3$ & 0\\ 
5   &    36         & 0  & 2  & $z^3+z$             &$0$          & $z^3+2z$ & $0,1,2$\\
$6\le m\le 22$& $m+31$ & $1,2,\dots, 17 $ & 2  & $z$  & 0                 & {$2z$} &0\\
$23\le m$ & $m+31$ & $0,1,\dots,17$ & $3,4,\dots$ & $z$ & 0              &  {$2z$} & 0
\end{tabular}
\end{center}
By \cite[Lemma 4.9]{pauli-sinclair} we have
\[
\varphi_{13,1}=\delta \overline{\textstyle\binom{13}{1}^{-1}(-\varphi_{0,1})^{0+1}3^0}=-\delta\varphi_{0,1} \text{ and }
\varphi_{3,1}=1\overline{\textstyle\binom{3}{3}^{-1}(-\varphi_{0,1})^{0+1}3^0}=-\varphi_{0,1}.
\]
Thus each of our extensions is generated by one polynomial of the form
\[
\varphi(x)=x^{18}+3\varphi_{17,1}x^{17}+3\varphi_{16,1}x^{16}+3\varphi_{14,1}x^{14}
-3\delta\varphi_{0,1}x^{13}-3\varphi_{0,1}x^3+3\varphi_{0,1}+9\varphi_{0,2}.
\]
When \(\delta=1\) we have \(\varphi_{0,2}=0\), 
\(\varphi_{17,1}\in\{0,1,2\}\), \(\varphi_{16}\in\{0,1,2\}\), \(\varphi_{14}\in\{0,1,2\}\), and \(\varphi_{0,1}\in\{1,2\}\) which yields $3^3\cdot 2=54$ polynomials.  Because none of these can be transformed into any of the others by changing the uniformizer, none of the corresponding extensions are isomorphic.  

When \(\delta=2\)
we have \(\varphi_{0,2}=\{0,1,2\}\), 
\(\varphi_{17,1}\in\{0,1,2\}\), \(\varphi_{16}\in\{0,1,2\}\), \(\varphi_{14}\in\{0,1,2\}\), 
and \(\varphi_{0,1}\in\{1,2\}\) which yields $3^4\cdot 2=162$ polynomials, no two of which define
isomorphic extensions.

\end{example}

For the extensions \(\KL/\Q_3\) with
\(f_{\KL/\Q_3}=2\), \(e_{\KL/\Q_3}=9\), and \(v_3(\disc(\KL/\Q_3))=15\)
there are two ramification polygons, see Figure \ref{fig ram pol examples}.
We consider one of them in the following.

\begin{example}[\href{https://www.lmfdb.org/padicField/family/3.2.9.30b}{LMFDB 3.2.9.30b}]\label{ex f=2 e=9 rel}
We construct all extensions \(\KL/\Q_3\) of degree 18
with \(v_3(\disc(\KL/\Q_3))=2(9+7-1)\) and ramification polygon 
\[
\{(-18,9),(-9,0),(-3,3),(-1,7)\}
\]
with slopes \(-1,\frac{1}{2},2\).
Let \(\Lur=\Q_3(\xi)\) where \(\xi\) is a root of the Conway polynomial 
\(\nu_2(x)=x^2+ 2x + 2\in\Z_3[x]\) be the unramified part of these extensions.
We define each extension by an Eisenstein polynomial \(\psi\in\OLur[x]\) which has 
ramification polygon 
\(\rampol=\{(-9,0),(-3,3),(-1,7)\}\).
By \cite[Proposition 3.10]{pauli-sinclair} this is equivalent to
\[
v_3(\psi_{i})=1 \text{ for } i\in\{3,7\},\,
v_3(\psi_{i})\ge 1 \text{ for } i \in \{6,8\},\,
v_3(\psi_{i})\ge 2 \text{ for } i \in \{1,2,4,5\}.
\]
Because \(\deg \psi=9\), to satisfy Condition \ref{con dist const} (1) we choose \(\psi_{0,1}=1\).  
With this assumption the polynomials \(\psi\) with distinct residual polynomials 
\[
(\overline{A}_{\frac{1}{2}},\overline{A}_2)=(z^3+\overline\gamma,\overline\gamma z^2+\overline\delta) \text{ where }\overline\gamma,\delta\in\F_9=\F_3(\overline\xi)
\] 
yield non-isomorphic extensions.
Using Proposition \ref{prop red} we get: 
\begin{center}
\begin{tabular}{c||c|c|c||c|c}
slope $m$ & $h=9\Phi_\rampol(m)$         & $i=h\bmod 9$   & $k=(h \bdiv 9)+1$ & $\Aadd_{m}$ & \(\varphi_{i,k}\) \\\hline
1 & 15 & 6 & 1 & $z^3$                        & 0 \\
2 & 18 & 0 & 2 & \(z\overline{A}_2\)
& \\
$3\le m \le 10 $ & $16+m$ & \(1,2,\dots,8\) & \(2\) & $z$ & 0 \\
$11\le m$ & $16+m$ & \(0,1,\dots,8\) & \(3,4,5,\dots\) & $z$ & 0 \\[1ex]
\end{tabular}
\end{center}
With the above we obtain the following generic polynomial (c.f. \cite[\S 3.3]{families})
\[
\psi(x)=x^9+3\psi_{8,1}x^8+3\psi_{7,1}x^7+3\psi_{3,1}x^3+3+9\psi_{0,2}.
\]
Adding \(\overline{A}_2\) to our considerations reduces our choice of coefficients further. 
By \cite[Lemma 4.9]{pauli-sinclair}:
\[
\overline\psi_{3,1}=\overline{\gamma7^{-1}(-\psi_{0,1})}=-\overline\gamma
\text{ and }\overline\psi_{7,1}=\overline{\delta7^{-1}(-\psi_{0,1})}=-\overline\delta\,\overline\xi^s.
\]
Thus
\(
\psi(x)=x^9 + 3\psi_{8,1} x^8- 3\delta x^7 - 3\gamma x^3 + 3 +9\psi_{0,2}
\)
where \(\psi_{8,1}, \psi_{0,2}\in\RepLur\).
Because \(\psi_{0}\sim 3\) we have \(\eta=3/\pi_L^9\sim -1\).
We have nonzero solutions for \(\Aadd_2(z)=\overline\gamma z^3+\overline\delta z=0\) 
when 
\(
\gamma/\delta\in\{1,2,1+\overline{\xi},2+2\overline{\xi}\}\) 
and get
\(\ker(z^3+z)=\{0,1+\overline{\xi},2+2\overline{\xi}\}\), 
\(\ker(z^3+(1+\overline{\xi})z)=\{0,1+2\overline{\xi},2+\overline{\xi}\}\), 
\(\ker(z^3+2z)=\{0,1,2\}\), and
\(\ker(z^3+(2+2\overline{\xi})z)=\{0,\overline{\xi},2\overline{\xi}\}\).
With our choice of representatives of the cokernels from Definition 
we obtain
\[
\overline{\psi}_{0,2}\in\left\{
\begin{array}{ll}
\{0,1,2\}&\text{if }\overline{\gamma/\delta}\in\{1,1+\overline{\xi},2+2\overline{\xi}\}\\  
\{0,\overline\xi,2\overline\xi\}&\text{if }\overline{\gamma/\delta}=2\\
\{0\}&\text{otherwise }
\end{array}
\right.
\]
Thus all extensions of degree 9 
over \(\Lur\)
with ramification polygon \(\rampol\)
are defined by exactly one of the 
\(9\cdot 8(3\cdot 3+1\cdot 3+4)=1152\) polynomials constructed above.
For each \(\delta\) there are \(9\) or \(9\cdot 3=27\) polynomials with residual invariant \([z^3+1,z^1+\overline\delta]\).

\end{example}


We now find the  distinguished defining {\O}ystein polynomial of one of the fields considered above over \(\Q_3\).

\begin{example}
\label{ex f=2 e=9 abs}
We find the distinguished defining polynomial of \(\KL/\Q_3\)
where \(\Lur=\Q_3(\xi)\) is the unramified extension of \(\Q_3\)
of degree 2 given by \(\nu(x) = x^2 + 2x + 2 \) with \(\nu(\xi)=0\)
and \(\KL=\Lur(\piL)\), where \(\piL\) is a root of
\[\psi(x) = x^9 + 3x^8 + 6x^7 + (6+6\xi)x^3 + 3.\]
The ramification polygon  
\[
\rampol=\{(-18,9),(-9,0),(-3,3),(-1,7)\}
\]
of \(\KL/\Q_3\) has slopes \(-1,\frac{1}{2},2\).
As in Example \ref{ex f=2 e=9 rel} we have
\[9\Phi_\rampol(1)=15,\,9\Phi_\rampol(2)=18,\text{ and }9\Phi_\rampol(m)=16+m \text{ for }m\ge3.\]
Thus, by Corollary \ref{coro easy} we can find a defining {\O}ystein polynomial 
\[
\varphi(x) = \nu(x)^9+\sum_{i=0}^8 \nu(x)^i\sum_{k=1}^\infty3^k\varphi_{i.k}^*(x)\in\Z_3[x]
\]
where \(\deg\varphi_{i,k}^*<2\), the coefficients of \(\varphi_{i,k}^*\) are taken from \(\{0,1,2\}\),
and 
\begin{equation}\label{eq ex red easy}
\varphi^*_{i,2}=0 \text{ for }1\le i\le 8\text{ and }
\varphi^*_{i,k}=0 \text{ for }3\le k,\; 0\le i\le 8.
\end{equation}
Let \(\piL\in\KL\) be a root of \(\psi\) and let \(\alpha\) be a root of \(\nu(x)-\piL\).  Then the characteristic polynomial of \(\alpha\) is {\O}ystein.   
Applying (\ref{eq ex red easy}) we get the defining polynomial
\[
\varphi_0(x)=\nu(x)^9 + 3\nu(x)^8 + (6x + 3)\nu(x)^7 + (3x + 3)\nu(x)^6 + (3x + 3)\nu(x)^4 + (6x + 6)\nu(x)^3 + 3
\]
Since \(\rampol\) has no segment of slope \(0\) and we already have \(\overline\varphi^*_{0,1}(\alpha)=\overline\psi_{0,1}=1\)
we do not need to go through the zeroth (\(m=0\)) reduction step.
For slope \(m=1\) we have \(e\Phi(m)=9\) and
\[
\Aadd_{1}(z)=\overline{\left(\frac{\rho(\nu(\alpha) z)}{\nu(\alpha)^{15}}\right)}=2z^3
\text{ where }\rho(x)=\varphi(\alpha+\alpha x)
\]
Since \(9\Phi_\rampol(1)=15\) and \(15\bmod 9=6\) and \(15\bdiv 9=1\) the change from minimal polynomial
of \(\alpha\) to minimal polynomial of \(\alpha+\theta\nu(\alpha)^2\) affects
the coefficient \(\varphi^*_{6,1}\).
Since \(\Aadd_{1}(z)=2z^3\) is surjective we can set \(\varphi^*_{6,1}\) to \(0\)
by finding \(\overline\theta\in\RLur\) with
\[
\Aadd_1(z)-\overline{\varphi^*_{6,1}(\xi)}
=2z^3-(\overline{\xi}+1)=0
\]
The only solution is \(z=\overline{\xi^6}\).  Thus we replace \(\varphi\) by the
minimal polynomial of \(\alpha+\xi^6\nu(\alpha)^2\) and after having applied the reductions for
\(m\ge 3\) given in (\ref{eq ex red easy}) obtain
\[
\varphi_1(x)=\nu(x)^9 + 6\nu(x)^7 + (3x + 3)\nu(x)^4 + (6x + 6)\nu(x)^3 + 18x + 12.
\]
For slope \(m=2\) we have \(9\Phi(m)=16\).
Because \(18\bmod 9=0\) and \(18\bdiv 9=2\) the change from minimal polynomial
of \(\alpha\) to minimal polynomial of \(\alpha+\theta\nu(\alpha)^2\) affects
the coefficient \(\varphi^*_{0,2}\) of \(\varphi\).
We have
\[
\Aadd_{2}(z)=\overline{\left(\frac{\rho(\nu(\alpha)^2 z)}{\nu(\alpha)^{16}}\right)}=2z^3 + (2\overline\xi+2)z
\text{ where }\rho(x)=\varphi(\alpha+\alpha x).
\]
The matrix representation of the \(\F_3\) linear map
\(\Aadd_2:\RLur \to \RLur\) is
\(\big(\begin{smallmatrix}
  1 & 2\\
  1 & 2
\end{smallmatrix}\big)\)
and its row echelon form is
\(\big(\begin{smallmatrix}
  1 & 2\\
  0 & 0
\end{smallmatrix}\big)\).
Reducing \(\overline\varphi^*_{0,2}=2x+1\) with the rows of the latter matrix over \(\F_3\) we get \((2x+1)-2(x+2)=0\).  Thus we can reduce \(\varphi^*_{0,2}\) to \(0\).
The solutions of \(\Aadd_2(z)=2\overline\xi+1\) are \(1\),\(\overline\xi\), and \(2\overline\xi+2\).
Because the characteristic polynomials of
\(\alpha+\nu(\alpha)^2\), \(\alpha+\xi\nu(\alpha)^2\), and \(\alpha+(2\xi+2)\nu(\alpha)^2\)
only differ at coefficients that can be set to 0 by (\ref{eq ex red easy}) the distinguished defining {\O}ystein polynomial of \(\KL/\Q_3\) is
\[
\varphi_2(x)=\nu(x)^9 + 6\nu(x)^7 + (3x + 3)\nu(x)^4 + (6x + 6)\nu(x)^3 + 3.
\]
\end{example}
}





\bibliography{local}
\bibliographystyle{amsalpha}

\end{document}